\documentclass[11pt,reqno]{amsart}
\usepackage{amsaddr}
\usepackage{graphicx}
\usepackage{amssymb}
\usepackage{amsfonts}
\usepackage[]{amsmath}
\usepackage[]{epsfig}
\usepackage[]{pstricks}
\newpsobject{malla}{psgrid}{subgriddiv=1,griddots=10,gridlabels=6pt}
\usepackage[]{float}
\usepackage{setspace}

\newpsobject{malla}{psgrid}{subgriddiv=1,griddots=10,gridlabels=6pt}
\newtheorem{theorem}{Theorem}[section]
\newtheorem{lemma}{Lemma}[section]
\newtheorem{definition}{Definition}[section]
\newtheorem{remark}{Remark}[section]

\numberwithin{equation}{section}
\newcommand{\R}{\mathbb R}

\begin{document}

\newpage
	\pagenumbering{arabic}	
\title[IBVP for some quadratic nonlinear Schr\"odinger equations on the half-line]{THE INITIAL-BOUNDARY VALUE PROBLEM FOR SOME QUADRATIC NONLINEAR SCHR\"ODINGER EQUATIONS ON THE HALF-LINE}

\author{M\'arcio Cavalcante}

\address{\emph{Instituto de Matem\'{a}tica, Universidade Federal do Rio de Janeiro\\ Ilha do Fund\~ao, 21945-970, Rio de Janeiro RJ-Brazil}}
\email{marciocavalcante@ufrj.br, mcavalcante@dim.uchile.cl}
\let\thefootnote\relax
\begin{abstract}
We prove local well-posedness for the initial-boundary value problem associated to some quadratic nonlinear Schr\"odinger equations on the half-line. The results are obtained in the low regularity setting by introducing an analytic family of boundary forcing operators, following the ideas developed in \cite{Holmerkdv}.
\end{abstract}
\maketitle
\section{Introduction}
This paper is concerned with the initial-boundary value problem (IBVP) on the right half-line for some quadratic nonlinear Schr\"odinger equations, namely
\begin{equation}\label{SK}
\begin{cases}
i\partial_tu-\partial_x^2 u=N_i(u,\overline{u}), & (x,t)\in(0,+\infty)\times(0,T),\ i=1,2,3,\\
u(x,0)=u_0(x),                                   & x\in(0,+\infty),\\
u(0,t)=f(t),                                     & t\in(0,T),
\end{cases}
\end{equation}
where $N_1(u,\overline{u})=u^2,\ N_2(u,\overline{u})=|u|^2$ or $N_3(\overline{u})=\overline{u}^2.$

The appropriate spaces for the initial and boundary data is motivated by the behavior of the solutions for linear Schr\"odinger equation in $\mathbb{R}$. Let $e^{-it\partial_x^2}$ the linear homogeneous solution group in $\mathbb{R}$ for the Schr\"odinger equation. The following smoothing effect 
\begin{equation*}
\|e^{-it\partial_x^2}\phi\|_{L_x^{\infty}\dot{H}^{\frac{2s+1}{4}}(\mathbb{R}_t)}\leq c\|\phi\|_{\dot{H}^{s}(\mathbb{R})},
\end{equation*}
can be found in \cite{KPV}, where this inequality is sharp in the sense that $\frac{2s+1}{4}$ cannot be replaced by any higher number. We are thus motivated to consider the IBVP (\ref{SK}) in the setting
\begin{equation}\label{regularidade}
u_0\in H^s(\mathbb{R}^+)\ \text{and}\ f(t)\in H^{\frac{2s+1}{4}}(\mathbb{R}^+).
\end{equation}

Our goal in studying \eqref{SK}-\eqref{regularidade} is to obtain low regularity results. Then we consider $s\leq 0$, where compatibility conditions between the values $u(0)$ and $f(0)$ are not required. 

The motivation to study the IBVP (\ref{SK})-\eqref{regularidade} come from the associated initial-value problem (IVP) in $\mathbb{R}$, that is, 
\begin{equation}\label{SKR}
\begin{cases}
i\partial_tu-\partial_x^2 u=N(u,\overline{u}),& (x,t)\in\mathbb{R}\times(0,T),\\
u(x,0)=u_0(x),& x\in\mathbb{R}.
\end{cases}
\end{equation}

The IVP \eqref{SKR} has been studied by several authors of the last decades (see e.g. the monograph by Linares and Ponce \cite{LP}). Here we mainly discuss the principal results for quadratic nonlinearities. Cazenave and Weissler \cite{CW} and Tsutsumi \cite{T} obtained local well-posedness (LWP), i.e., existence, uniqueness and continuity of the data-to-solution map, in $H^s(\mathbb{R})$ for $s\geq 0$, for any quadratic nonlinear term, i.e., $|N(u,\overline{u})|\leq c |u|^2$. The proof is based on the version of the Strichartz estimate for
the free Schr\"odinger group $e^{-it\partial_x^2}$ found in \cite{GTV2}. Kenig, Ponce  and Vega \cite{KPV2-a} showed LWP in $H^s(\mathbb{R})$ for $s\in(-\frac{3}{4},0]$ for the nonlinear terms $u^2$ or $\overline{u}^2$, and  $s\in(-\frac{1}{4},0]$ for nonlinearity $|u|^2$, by using Fourier restriction norm method introduced by Bourgain in \cite{Bourgain1}. Bejenaru and Tao \cite{bt} showed LWP in $H^s(\mathbb{R})$ for $s\geq -1$ when $N(u,\overline{u})=u^2$, by an iteration using a modification of the standard Bourgain spaces. These authors also proved local ill-posedness (LIP) for $s<-1$, in the sense that the data-to-solution map fails to be continuous. In \cite{Kishimoto} Kishimoto obtained LWP in $H^{s}(\mathbb{R})$ for $s\geq -1$ and LIP for $s<-1$ when $N(u,\overline{u})=\overline{u}^2$.

The solutions we construct to the IBVP \eqref{SK}-\eqref{regularidade} shall have the following properties.
\begin{definition} A function $u(x,t)$ will be called a distributional solution of (\ref{SK}) on $[0,T]$ if
	\begin{itemize}
		\item[(a)] \textbf{Well-defined nonlinearity}: $u$ belongs to some space $X$ with the property that $u\in X$ implies $N_i(u,\overline{u})$, $i=1,2,3$ is a well-defined distribution;
		\item [(b)]$u(x,t)$ satisfies the equation (\ref{SK}) in the sense of distributions on the set $(x,t)\in(0,+\infty)\times(0,T)$;
		\item [(c)]\textbf{Space traces}: $u\in C\big([0,T];\,H^s(\mathbb{R}_x^+)\big)$ and in this sense $u(\cdot,0)=u_0$;
		\item [(d)]\textbf{Time traces}: $u\in C\big(\mathbb{R}_x^{+};\,H^{\frac{2s+1}{4}}(0,T)\big)$ and in this sense $u(0,\cdot)=f$.
	\end{itemize}
\end{definition}
In our case, $X$ shall be the restriction on $\mathbb{R}^+\times(0,T)$ of the functions belonging to the Space $X^{s,b}\cap C\big(\mathbb{R}_t;\,H^s(\mathbb{R}_x)\big)\cap C\big(\mathbb{R}_x;\,H^{\frac{2s+1}{4}}(\mathbb{R}_t)\big)$, where $X^{s,b}$ is the Bourgain space, where
\begin{equation*}
\|u\|_{X^{s,b}}=\left(\int\int \langle \xi\rangle^{2s} \langle \tau-\xi^2\rangle^{2b} |\hat{u}(\xi,\tau)|^2d\xi d\tau\right)^{\frac{1}{2}}.
\end{equation*} 
The space $X^{s,b}$, with $b>\frac{1}{2}$, is typically employed in the study of the IVP \eqref{SKR}.  As in the work of Colliander and Kenig \cite{CK}, to handle with the  Duhamel boundary forcing operator in our study of the IBVP, we need to take $b<\frac{1}{2}$, force us to intercept $X^{s,b}$ with the space $C\big(\mathbb{R}_t;\,H^s(\mathbb{R}_x)\big)$, because for $b<\frac{1}{2}$, $X^{s,b}$ is not embedded into $C\big(\mathbb{R}_t;\,H^s(\mathbb{R}_x)\big)$. An important point here is that for $b>\frac{1}{2}$ the bilinear estimates in Bourgain Spaces $X^{s,b}$ were already obtained in \cite{KPV2-a}, while here we need to prove these estimates in the case $b<\frac{1}{2}$.
 
Our main results are summarized in the following theorem.
\begin{theorem}\label{teorema}
	Let $s_i=-3/4$ for $i=1,3$ and $s_2=-1/4$. For any initial-boundary data $(u_0,f)\in H^{s}(\mathbb{R}^+)\times H^{\frac{2s+1}{4}}(\mathbb{R}^+)$, with 
	$s\in (s_i,0]$, there exists a positive time $T_i$, depending only of $\|u_0\|_{H^s(\mathbb{R}^+)}$ and $\|f\|_{H^{\frac{2s+1}{4}}(\mathbb{R}^+)}$, and a distributional solution $u_i(x,t)$ of the IBVP \eqref{SK}-\eqref{regularidade} with nonlinearity $N_i(u,\overline{u})$, which is defined in $C\bigl([0,T];\; H^s(\mathbb{R}^+)\bigr)\cap C\bigl(\mathbb{R}^+;\; H^{\frac{2s+1}{4}}(0,T)\bigr)$. Moreover, the map $(u_0,f)\longmapsto u_i$ from $H^s(\mathbb{R}^+)\times H^{\frac{2s+1}{4}}(\mathbb{R}^+)$ into  $C\big([0,T_i];\,H^s(\mathbb{R}_x^+)\big)\cap C\big(\mathbb{R}_x^{+};\,H^{\frac{2s+1}{4}}(0,T)\big)$ is analytic.
\end{theorem}
As regards the question of regularity this result is similar of the classical results obtained by Kenig, Ponce and Vega in \cite{KPV2-a} for the pure IVP \eqref{SKR} on $\mathbb{R}$. 

The proof of Theorem \ref{teorema} involves the approach of Colliander and Kenig \cite{CK}, in their treatment of the generalized Korteweg-de Vries on the half-line. In this work they introduced a new method to solve IBVP for nonlinear
dispersive partial differential equations by recasting these problems as IVP
with an appropriate forcing term. We also use the ideas contained in Holmer's works \cite{Holmer} and \cite{Holmerkdv}, in his study of nonlinear Schr\"odinger and Korteweg-de Vries (KdV) equations on the half-line, by adapting the method of \cite{CK}. 
  
The method consists in solving a forced IVP in $\R$, analogous to the \eqref{SK}, 
\begin{equation}\label{SKF}
\left \{
\begin{array}{l}
i\partial_t\tilde{u}-\partial_x^2 \tilde{u}=N(\tilde{u},\overline{\tilde{u}})+D(x)h(t);\ (x,t)\in\mathbb{R}\times(0,T),\\
\tilde{u}(x,0)=\tilde{u}_0(x),\ x\in\mathbb{R}.
\end{array}
\right.
\end{equation}
where $\tilde{u}_0$ is a nice extension of $u_0$ and $D$ is a distribution supported in $\mathbb{R}^-$. The boundary forcing function $h(t)$ is selected to ensure that
\begin{equation}\label{priori}
\tilde{u}(0,t)=f(t),\ t\in (0,T).
\end{equation}
Upon constructing the solution $\tilde{u}$ of (\ref{SKF}), we obtain a distributional solution of (\ref{SK}), by restriction, as $u=\tilde{u}|_{\{\mathbb{R}^+\times (0,T)\}}$.

The solution of \eqref{SKF} satisfying \eqref{priori} is constructed using the classical restricted norm method of Bourgain (see \cite{Bourgain1} and \cite{KPV2-a}) and the inversion of a Riemann-Liouville fractional integration operator.

  The crucial point in this work is the study of the Duhamel boundary forcing operator of \cite{Holmer},  
\begin{equation*}\label{0501}
\mathcal{L}f(x,t)=2e^{-i\frac{\pi}{4}}\int_0^te^{i(t-t')\partial_x^2}\delta_0(x)\mathcal{I}_{-\frac{1}{2}}f(t')dt',
\end{equation*} 
where $\mathcal{I}_{-\frac{1}{2}}$ is the the Riemann-Liouville fractional integral of order $-1/2$ given by
\begin{equation*}
\mathcal{I}_{-\frac{1}{2}}f(t')=\frac{1}{\Gamma(1/2)}\int_0^{t}(t-s)^{-1/2}f'(s)ds,\ \text{for}\ f\in C_0^{\infty}(\R^+),
\end{equation*} 
in the context of Bourgain Spaces. In \cite{Holmer} the operator $\mathcal{L}$ was studied in the context of Strichartz estimates, to solve the IBVP associated to the classical nonlinear Schr\"odinger equation, with nonlinearity $\lambda|u|^{\alpha-1}u$. We will obtain the Bourgain spaces estimates,
\begin{equation*}
\|\psi(t)\mathcal{L}f(x,t)\|_{X^{s,b}}\leq c \|f\|_{H_0^{\frac{2s+1}{4}}(\R^{+})},
\end{equation*}
 for $-\frac{1}{2}<s<\frac{1}{2}$, where $\psi(t)$ is a smooth cutoff function. To extend this estimate for others values of $s$ we define an analytic family of boundary forcing operators, analogous to the one in \cite{Holmerkdv} in the study of the IBVP for the KdV equation on the half-line, in the low regularity setting. This family of operator generalizes the single operator $\mathcal{L}$ of \cite{Holmer}.
 
We have applied the same technical to study the IBVP associated to the Schr\"odinger-Korteweg-de Vries system on the half-line \cite{CM}.

We do not explore here the uniqueness of solutions obtained in Theorem \ref{teorema}. Actually we have uniqueness in the sense of Kato (see \cite{Kato}) only for a reformulation of the IBVP \eqref{SK} as an integral equation posed in $\R$, such that solves \eqref{SK}.  As there are many ways to transform the IBVP \eqref{SK} into an integral equation, we do not have uniqueness in the strong sense. We believe that the method given by Bona, Sun and Zhang in \cite{Bona} can be applied here. These authors introduced the concept of mild solutions, to treat the question of uniqueness for the Korteweg-de-Vries equation on the half-line. In \cite{Bona1} the same authors applied this method to get uniqueness of the solutions for the classical nonlinear Schr\"odinger equation, with nonlinearity $|u|^{\lambda-1}u$ on the half-line and on the bounded interval.

This paper is organized as follows: in the next section, we discuss some notation, introduce function spaces and recall some needed properties of these function spaces, and review the definition and basic properties of the Riemann-Liouville fractional integral. Sections \ref{section3}, \ref{section4} and \ref{section5}, respectively, are devoted to the needed estimates for the linear group, the Duhamel boundary forcing operators and the Duhamel
inhomogeneous solution operator. In Section \ref{section6} we define the Duhamel boundary forcing operators class, adapted from \cite{Holmerkdv}, and prove the needed estimates for it. In Section \ref{bilinear}, we prove the bilinear estimates. Finally, Section \ref{section8} is devoted to prove Theorem \ref{teorema}.



\section{Notations, Function Spaces, Riemman-Liouville Fractional Integral}
\subsection{Notations}For $\phi=\phi(x)\in S(\mathbb{R})$,  $\hat{\phi}(\xi)=\int e^{-i\xi x}\phi(x)dx$ will denote the Fourier transform of $\phi$. For $u=u(x,t)\in S(\mathbb{R}^2)$, $\hat{u}=\hat{u}(\xi,\tau)=\int e^{-i(\xi x+\tau t)}u(x,t)dxdt$ will denote its space-time Fourier transform, $\mathcal{F}_{x}u(\xi,t)$ denotes its space Fourier transform and $\mathcal{F}_{t}u(x,\tau)$ its time Fourier transform. Let $\langle \xi \rangle=1+|\xi|$. Define $(\tau-i0)^{\alpha}$ as the limit, in the sense of distributions, of $(\tau-\gamma i)^{\alpha} $, when $\gamma\rightarrow 0^{-}$. $\chi_{A}$ denotes the characteristic function of an arbitrary set $A$.

\subsection{Function Spaces}
For $s\geq 0$ define $\phi \in H^s(\mathbb{R}^+)$ if exists $\tilde{\phi}$ such that $\phi(x)=\tilde{\phi}(x)$ for $x>0$, in this case we set $\|\phi\|_{H^s(\mathbb{R}^+)}=\inf_{\tilde{\phi}}\|\tilde{\phi}\|_{H^{s}(\mathbb{R})}$. For $s\geq 0$ define $$H_0^s(\mathbb{R}^+)=\{\phi \in H^{s}(\mathbb{R}^+);\,\text{supp}\ \phi\ \subset[0,\infty) \}.$$ For $s<0$, define $H^s(\mathbb{R}^+)$ as the dual space to $H_0^{-s}(\mathbb{R}^+)$, and define $H_0^s(\mathbb{R}^+)$ as the dual space to $H^{-s}(\mathbb{R}^+)$.

Also define the $C_0^{\infty}(\mathbb{R}^+)=\{\phi\in C^{\infty}(\mathbb{R});\, \text{supp}\ \phi\subset [0,\infty)\}$ and $C_{0,c}^{\infty}(\mathbb{R}^+)$ as those members of $C_0^{\infty}(\mathbb{R}^+)$ with compact support. We remark that $C_{0,c}^{\infty}(\mathbb{R}^+)$ is dense in $H_0^s(\mathbb{R}^+)$ for all $s\in \mathbb{R}$.

Throughout the paper, we fix a cutoff function $\psi \in C_0^{\infty}(\mathbb{R})$ such that $\psi(t)=1$ if $t\in[-1,1]$ and supp $\psi\subset [-2,2]$. 

The following lemmas state elementary properties of the Sobolev spaces. For the proofs we refer the reader \cite{CK}.\begin{lemma}\label{sobolevh0}
	For $-\frac{1}{2}<s<\frac{1}{2}$ and $f\in H^s(\mathbb{R})$, we have
$\|\chi_{(0,+\infty)}f\|_{H^s(\mathbb{R})}\leq c \|f\|_{H^s(\mathbb{R})}.$
	\end{lemma}
	
\begin{lemma}\label{sobolev0}
	If $0\leq s<\frac{1}{2}$, then $\|\psi f\|_{H^s(\mathbb{R})}\leq c \|f\|_{\dot{H}^{s}(\mathbb{R})}$ and $\|\psi f\|_{\dot{H}^{-s}(\mathbb{R})}\leq c \|f\|_{H^{-s}(\mathbb{R})}$ , where $c$ only depends of $s$ and $\psi$. 
\end{lemma}
\begin{lemma}\label{cut}
Let $f\in  H_0^s(\mathbb{R}^+)$, with $-\infty<s<+\infty$. Then $
	\|\psi f\|_{H_0^s(\mathbb{R}^+)}\leq c \|f\|_{H_0^s(\mathbb{R}^+)}.
	$
\end{lemma}
For $s,b\in \mathbb{R}$, we introduce the classical Bourgain spaces $X^{s,b}$ and the modified Bourgain spaces $W^{s,b}$ related to the Schr\"odinger equation as the completion of $S'(\mathbb{R}^2)$ under the norms

\begin{equation*}
\|u\|_{X^{s,b}}=\left(\int\int \langle \xi\rangle^{2s} \langle \tau-\xi^2\rangle^{2b} |\hat{u}(\xi,\tau)|^2d\xi d\tau\right)^{\frac{1}{2}}
\end{equation*}
\begin{equation*}
\|u\|_{W^{s,b}}=\left(\int\int \langle \tau\rangle^{s} \langle \tau-\xi^2\rangle^{2b} |\hat{u}(\xi,\tau)|^2d\xi d\tau\right)^{\frac{1}{2}}.
\end{equation*}
\subsection{Riemman-Liouville fractional integral}
The tempered distribution $\frac{t_+^{\alpha-1}}{\Gamma(\alpha)}$ is defined as a locally integrable function for Re $\alpha>0$ by
\begin{equation*}
\left \langle \frac{t_+^{\alpha-1}}{\Gamma(\alpha)},\ f \right \rangle=\frac{1}{\Gamma(\alpha)}\int_0^{+\infty} t^{\alpha-1}f(t)dt.
\end{equation*}
For Re $\alpha>0$, we have that
\begin{equation*}
\frac{t_+^{\alpha-1}}{\Gamma(\alpha)}=\partial_t^k\left( \frac{t_+^{\alpha+k-1}}{\Gamma(\alpha+k)}\right),
\end{equation*}
for all $k\in\mathbb{N}$. This expression can be used to extend the definition, in the sense of distributions) of $\frac{t_+^{\alpha-1}}{\Gamma(\alpha)}$ to all $\alpha \in \mathbb{C}$.

A change of contour  calculation shows that
\begin{equation}\label{transformada}
\left(\frac{t_+^{\alpha-1}}{\Gamma(\alpha)}\right)^{\widehat{}}(\tau)=e^{-\frac{1}{2}\pi\alpha}(\tau-i0)^{-\alpha},
\end{equation}
where $(\tau-i0)^{-\alpha}$ is the distributional limit. If $f\in C_0^{\infty}(\mathbb{R}^+)$, we define
\begin{equation*}
\mathcal{I}_{\alpha}f=\frac{t_+^{\alpha-1}}{\Gamma(\alpha)}*f.
\end{equation*}
Thus, when Re $\alpha>0$,
\begin{equation*}
\mathcal{I}_{\alpha}f(t)=\frac{1}{\Gamma(\alpha)}\int_0^t(t-s)^{\alpha-1}f(s)ds,
\end{equation*}
and $\mathcal{I}_0f=f,\ \mathcal{I}_1f(t)=\int_0^tf(s)ds,$ and $\mathcal{I}_{-1}f=f'$. Also $\mathcal{I}_{\alpha}\mathcal{I}_{\beta}=\mathcal{I}_{\alpha+\beta}.$

The following lemmas, whose proofs can be found in \cite{Holmerkdv}, state some useful properties of the Riemman-Liouville fractional integral operator.
\begin{lemma}
	If $f\in C_0^{\infty}(\mathbb{R}^+)$, then $\mathcal{I}_{\alpha}f\in C_0^{\infty}(\mathbb{R}^+)$, for all $\alpha \in \mathbb{C}$.
\end{lemma}
\begin{lemma}\label{lio}
	If $0\leq \alpha <\infty$ and $s\in \mathbb{R}$, then $\|\mathcal{I}_{-\alpha}h\|_{H_0^s(\mathbb{R}^+)}\leq c \|h\|_{H_0^{s+\alpha}(\mathbb{R}^+)}.$
\end{lemma}
\begin{lemma}
	If $0\leq \alpha <\infty$, $s\in \mathbb{R}$ and $\mu\in C_0^{\infty}(\mathbb{R})$, then
$\|\mu\mathcal{I}_{\alpha}h\|_{H_0^s(\mathbb{R}^+)}\leq c \|h\|_{H_0^{s-\alpha}(\mathbb{R}^+)},$ where $c=c(\mu)$.
\end{lemma}
For more details on the distribution $\frac{t_+^{\alpha-1}}{\Gamma(\alpha)}$ see \cite{Fr}.
\section{Linear Version}\label{section3}
We define the unitary group associated to the linear Schr\"odinger equation as
\begin{equation*}
	e^{-it\partial_x^2}\phi(x)=\frac{1}{2\pi}\int e^{ix\xi}e^{it\xi^2}\hat{\phi}(\xi)d\xi,
\end{equation*}
it follows that 
\begin{equation}\label{linear}
\begin{cases}
		(i\partial_t-\partial_x^2)e^{-it\partial_x^2}\phi(x) =0,& (x,t)\in\mathbb{R}\times\mathbb{R},\\
		e^{-it\partial_x^2}\phi(x)\big|_{t=0}=\phi(x),& x\in\mathbb{R}.
\end{cases}
\end{equation}
\begin{lemma}\label{grupo}
	Let $s\in\mathbb{R}$, $0< b<1$. If $\phi\in H^s(\mathbb{R})$, then
	\begin{itemize}
		\item[(a)](Space traces) $\|e^{-it\partial_x^2}\phi(x)\|_{C\big(\mathbb{R}_t;\,H^s(\mathbb{R}_x)\big)}\leq c\|\phi\|_{H^s(\mathbb{R})}$;
		\item[(b)] (Time traces) $\|\psi(t) e^{-it\partial_x^2}\phi(x)\|_{C(\mathbb{R}_x;H^{\frac{2s+1}{4}}(\mathbb{R}_t))}\leq c \|\phi\|_{H^s(\mathbb{R})}$;
		\item [(c)](Bourgain spaces) $\|\psi(t)e^{-it\partial_x^2}\phi(x)\|_{X^{s,b}}\leq c\|\psi\|_{H^1(\mathbb{R})} \|\phi\|_{H^s(\mathbb{R})}$.
	\end{itemize}
\end{lemma}
\begin{proof}
The assertion in (a) follows  from properties of group $e^{-it\partial_x^2}$, (b) was obtained in \cite{KPV} and the proof of (c) can be found in \cite{GTV}.
\end{proof}
\section{The Duhamel Boundary Forcing Operator}\label{section4}
We now introduce the Duhamel boundary forcing operator of Holmer \cite{Holmer}. For $f\in C_0^{\infty}(\mathbb{R}^+)$, define the boundary forcing operator
\begin{eqnarray*}
	\mathcal{L}f(x,t)&=&2e^{-i\frac{\pi}{4}}\int_0^te^{-i(t-t')\partial_x^2}\delta_0(x)\mathcal{I}_{-\frac{1}{2}}f(t')dt'\\
	&=&\frac{1}{\sqrt{\pi}}\int_0^t(t-t')^{-\frac{1}{2}}e^{-\frac{ix^2}{4(t-t')}}\mathcal{I}_{-\frac{1}{2}}f(t')dt'.
\end{eqnarray*}
The equivalence of the two definitions is clear from the formula
\begin{equation*}
\mathcal{F}_x\left( \frac{e^{-i \frac{\pi}{4}\text{sgn}\ a} }{2|a|^{1/2}\sqrt{\pi}}e^{\frac{ix^2}{4a}}\right)(\xi)=e^{-ia\xi^2},\ \forall\ a\in \mathbb{R}.
\end{equation*}
From this definition, we see that
\begin{equation}\label{forçante00}
	\begin{cases}
		(i\partial_t-\partial_x^2)\mathcal{L}f(x,t)=2e^{i\frac{\pi}{4}}\delta_0(x)\mathcal{I}_{-\frac{1}{2}}f(t),& x,t\in\mathbb{R},\\
		\mathcal{L}f(x,0)=0,& x\in \R.
	\end{cases}
\end{equation}
The following lemma, whose proof can be found in \cite{Holmer}, states some continuity properties of $\mathcal{L}f(x,t)$.
\begin{lemma}\label{continuidade}
Let $f\in C_{0,c}^{\infty}(\mathbb{R}^+)$.
	\begin{itemize}
		\item[(a)] For fixed $t$, $\mathcal{L}f(x,t)$  is continuous in $x$ for all $x\in\mathbb{R}$ and $\partial_x\mathcal{L}f(x,t)$ is continuous in $x$ for $x\neq 0$ with
		\begin{equation*}
		\lim_{x\rightarrow 0^{-}}\partial_x\mathcal{L}f(x,t)=-e^{i\frac{\pi}{4}}\mathcal{I}_{-1/2}f(t),\ \lim_{x\rightarrow 0^{+}}\partial_x\mathcal{L}f(x,t)=e^{i\frac{\pi}{4}}\mathcal{I}_{-1/2}f(t).
		\end{equation*}
		\item[(b)] For $N,k$ nonnegative integers and fixed $x$, $\partial_t^k\mathcal{L}f(x,t)$  is continuous in $t$ for all $t\in \mathbb{R^+}$. We also have the pointwise estimates, on $[0,T]$,
		\begin{equation*}
		|\partial_t^k \mathcal{L}f(x,t)|+|\partial_x\mathcal{L}f(x,t)|\leq c\langle x\rangle^{-N},
		\end{equation*}
	\end{itemize}
\end{lemma}
where $c=c(f,N,k,T)$.

Let $f\in C_0^{\infty}(\R^+)$, set $u(x,t)=e^{-it\partial_x^2}\phi(x)+\mathcal{L}(f-e^{-it\partial_x^2}\phi(x)\big|_{x=0})$. Then, by Lemma \ref{continuidade} (a) $u(x,t)$ is continuous in $x$. Thus $u(0,t)=f(t)$ and $u(x,t)$ solves the problem
\begin{equation}\label{forçante}
	\begin{cases}
		(i\partial_t-\partial_x^2)u(x,t)=2e^{i\frac{\pi}{4}}\delta(x)\mathcal{I}_{-1/2}(f-e^{-it\partial_x^2}\phi(x)\big|_{x=0}),& (x,t)\in\mathbb{R},\\
		u(x,0)=\phi(x),& x\in\mathbb{R},\\
		u(0,t)=f(t),& t\in \R.
\end{cases}
\end{equation}
This would suffice to solve the linear analogue of the half-line problem.

\section{Nonlinear versions}\label{section5}
We define the Duhamel inhomogeneous solution operator $\mathcal{D}$ as
\begin{equation*}
\mathcal{D}w(x,t)=-i\int_0^te^{-i(t-t')\partial_x^2}w(x,t')dt',
\end{equation*}
it follows that
\begin{equation}\label{nonlinear}
\left \{
\begin{array}{l}
(i\partial_t-\partial_x^2)\mathcal{D}w(x,t) =w(x,t),\ (x,t)\in\mathbb{R}\times\mathbb{R},\\
\mathcal{D}w(x,0) =0,\ x\in\mathbb{R}.
\end{array}
\right.
\end{equation}
\begin{lemma}\label{duhamel}
	\begin{itemize}Let $s\in \mathbb{R}$. Then:
		\item[(a)](Space traces) If $-\frac{1}{2}<c<0$ , then
		\begin{equation*}
		\|\psi(t)\mathcal{D}w(x,t)\|_{C\big(\mathbb{R}_t;\,H^s(\mathbb{R}_x)\big)}\leq c\|w\|_{X^{s,c}};
		\end{equation*}
		\item[(b)](Time traces) If $-\frac{1}{2}<c<0$, then
		 \begin{equation*}
		\|\psi(t)\mathcal{D}w(x,t)\|_{C(\mathbb{R}_x;H^{\frac{2s+1}{4}}(\mathbb{R}_t))}\leq 
		\begin{cases}
		c\|w\|_{X^{s,c}},& \text{if}\ -\frac{1}{2}< s\leq\frac{1}{2},\\
		c(\|w\|_{W^{s,c}}+\|w\|_{X^{s,c}}),& \text{for\ all}\  s\in\mathbb{R};
		\end{cases}
		\end{equation*}
		\item[(c)](Bourgain spaces estimates) If $-\frac{1}{2}<c\leq0\leq b\leq c+1$, then
		\begin{equation*}
		\|\psi(t)\mathcal{D}w(x,t)\|_{X^{s,b}}\leq \|w\|_{X^{s,c}}.
		\end{equation*}
	\end{itemize}
\end{lemma}
\begin{remark}
	We note that the  $W^{s,b}$ (time-adapted) Bourgain spaces, used in Lemma \ref{duhamel} (b), are introduced in order to cover the full values of $s$.
	\end{remark}
	\begin{proof}
		
To prove (a), we use $2\chi_{(0,t)}(t')=sgn (t')+sgn(t-t')$, then 
	\begin{equation*}
	\begin{split}
		 &\psi(t)\int_0^t e^{i(t-t')\partial_x^2}w(\cdot,t')dt'=\psi(t)\int_{\xi}e^{ix\xi}\int_{\tau}\frac{e^{it\tau}-e^{it\xi^2}}{\tau-\xi^2}\hat{w}(\xi,\tau)d\tau d\xi\\
		 &= \psi(t)\int_{\xi}e^{ix\xi}\int_{|\tau-\xi^2|\leq 1}\frac{e^{it\tau}-e^{it\xi^2}}{\tau-\xi^2}\hat{w}(\xi,\tau)d\tau d\xi\!+\!\psi(t)\int_{\xi}e^{ix\xi}\int_{|\tau-\xi^2|>1}\frac{e^{it\tau}-e^{it\xi^2}}{\tau-\xi^2}\hat{w}(\xi,\tau)d\tau d\xi\\
		& :=w_1+w_2.
		\end{split}
	\end{equation*}
	To estimate $w_1$, we use $\left|\psi(t)\frac{e^{it\tau}-e^{it\xi^2}}{\tau-\xi^2}\right|\leq c$, when $|\tau-\xi^2|\leq 1$, more the Cauchy-Schwarz inequality. 
	
To estimate $w_2$, we use $\left|\psi(t)\frac{e^{it\tau}-e^{it\xi^2}}{\tau-\xi^2}\right|\leq c\langle \tau -\xi^2\rangle^{-1}$, when $|\tau-\xi^2|> 1$, and the Cauchy-Schwarz inequality. 
	To prove (b), we take $\theta(\tau)\in C^{\infty}(\mathbb{R})$ such that $\theta(\tau)=1$ for $|\tau|<\frac{1}{2}$ and supp $ \theta \subset[-\frac{2}{3},\frac{2}{3}]$, then
\begin{eqnarray*}
		& &\mathcal{F}_x\left( \psi(t)\int_0^te^{i(t-t')\partial_x^2}w(x,t')\right)(\xi)=\psi(t)\int_{\tau}\frac{e^{it\tau}-e^{it\xi^2}}{\tau-\xi^2}\hat{w}(\xi,\tau)d\tau
		\\
		& &\quad= \psi(t)e^{it\xi^2}\int_{\tau}\frac{e^{it(\tau-\xi^2)}-1}{\tau-\xi^2}\theta(\tau-\xi^2)\hat{w}(\xi,\tau)d\tau+\psi(t)\int_{\tau}e^{it\tau}\frac{1-\theta(\tau-\xi^2)}{\tau-\xi^2}\hat{w}(\xi,\tau)d\tau\\
		& &\quad \quad -\psi(t)e^{it\xi^2}\int_{\tau}\frac{1-\theta(\tau-\xi^2)}{\tau-\xi^2}\hat{w}(\xi,\tau)d\tau:=\mathcal{F}_xw_1+\mathcal{F}_xw_2-\mathcal{F}_xw_3.
	\end{eqnarray*}
	By the power series expansion for $e^{it(\tau-\xi^2)}$, $w_1(x,t)=\displaystyle\sum_{k=1}^{\infty}\frac{\psi_k(t)}{k!}e^{-it\partial_x^2}\phi_k(x),$ 
	where $\psi_k(t)=i^kt^k\theta(t)$ and $\hat{\phi}_k(\xi)=\int_{\tau}(\tau-\xi^2)^{k-1}\theta(\tau-\xi^2)\hat{w}(\xi,\tau)d\tau.$ Using Lemma \ref{grupo} (b), it suffices to show that $\|\phi_k\|_{H^s(\mathbb{R})}\leq c\|u\|_{X^{s,c}}$.
Using the definition of $\phi_k$ and the Cauchy-Schwarz inequality,
	\begin{eqnarray*}
		\|\phi_k\|_{H^s(\mathbb{R}_x)}^2&=&c\int_{\xi}\langle \xi \rangle^{2s}\left( \int_{\{\tau:|\tau-\xi^2|\leq\frac{2}{3}\}}\sum_{k=1}^{\infty}(\tau-\xi^2)^{k-1}\theta(\tau-\xi^2)\hat{u}(\xi,\tau)\right)^2d\xi\\
		&\leq&c\int_{\xi}\langle \xi \rangle^{2s}\int_{\tau}\langle \tau-\xi^2\rangle^{2c}|\hat{u}(\xi,\tau)|^2d\tau d\xi.
	\end{eqnarray*}
	This completes the estimate for $w_1$. To treat $w_2$, we change variables $\eta=\xi^2$, then 
	\begin{equation}\label{28101}
		\|w_2\|_{C\big(\mathbb{R}_x;\,H^{\frac{2s+1}{4}}(\mathbb{R}_t)\big)}\leq c\int_{\tau}\langle \tau\rangle^{\frac{2s+1}{2}}\left( \int_{\eta=0}^{+\infty}\langle \tau-\eta\rangle^{-1}\eta^{-\frac{1}{2}}\big|\hat{w}(\pm\eta^{\frac{1}{2}},\tau)\big|d\eta\right)^2d\tau.
	\end{equation}
	By Cauchy-Schwarz \eqref{28101} is bounded by 
	\begin{equation*}
c\int_{\tau}\langle \tau\rangle^{\frac{2s+1}{2}}G(\tau)\int_{\eta=0}^{\infty}\langle \tau-\eta\rangle^{-2b}|\eta|^{-\frac{1}{2}}\big|\hat{w}(\pm\eta^{\frac{1}{2}},\tau)\big|^2d\eta d\tau,
	\end{equation*}
	where $G(\tau)=\int_{\eta}\langle \tau-\eta\rangle^{-2+2b}|\eta|^{-\frac{1}{2}}\langle \eta\rangle^{-s} d\eta.$ Separating $G(\tau)$ in the regions $|\eta|\leq 1,\ 2|\eta|\leq |\tau|,\ |\tau|\leq 2|\eta|$, we can obtain $G(\tau)\leq c \langle\tau\rangle^{-\frac{1}{2}}$.

	Now suppose $-\frac{1}{2}< s\leq \frac{1}{2}$. Applying the Cauchy-Schwarz inequality we obtain

\begin{equation*}
\|w_2\|_{C\big(\mathbb{R}_x;\,H^{\frac{2s+1}{4}}(\mathbb{R}_t)\big)}\leq\int_{\tau}\langle \tau\rangle^{\frac{2s+1}{2}}G(\tau)\int_{\xi}\langle \tau+\xi^2\rangle^{2c}\langle\xi\rangle^{2s}|\hat{w}_2(\xi,\tau)|^2d\xi d\tau,
\end{equation*}

where $G(\tau)= c\int_{\eta}\langle \tau-\eta\rangle^{-2-2c}|\eta|^{-\frac{1}{2}}\langle\eta\rangle^{-s} d\eta.$

Separating in the regions  
	$|\eta|\leq 1,\ 2|\eta|\leq |\tau|,\ |\tau|\leq 2|\eta|$ and using that $-\frac{1}{2}<s\leq \frac{1}{2}$, we can obtain $G(\tau)\leq c \langle\tau \rangle^{-\frac{2s+1}{2}}$. This completes the estimate for $w_2$.
	To obtain $w_3$, we write $w_3=\psi(t)e^{-it\partial_x^2}\phi(x)$, where $\hat{\phi}(\xi)=\int\frac{1-\theta(\tau-\xi^2)}{\tau-\xi^2}\hat{w}(\xi,\tau)d\tau$. Using Lemma \ref{grupo} (b) and the Cauchy-Schwarz inequality, we obtain
	\begin{eqnarray*}
	\|w_3\|_{C(\mathbb{R}_x;H^{\frac{2s+1}{4}}(\mathbb{R}_t))}&=&c\|\psi(t)e^{-it\partial_x^2}\phi(x)\|_{C(\mathbb{R}_x;H^{\frac{2s+1}{4}}(\mathbb{R}_t))}\leq c\|\phi\|_{H^s(\mathbb{R})}\\
&\leq&c\int_{\xi}\langle \xi \rangle^{2s}\left(\int_{\tau}|\hat{w}(\xi,\tau)|^2\langle\tau-\xi^2\rangle^{2c}d\tau \int \frac{d\tau}{\langle\tau-\xi^2\rangle^{2-2c}}\right) d\xi.
	\end{eqnarray*}Since $c>-\frac{1}{2}$, we have $\int \frac{1}{\langle\tau-\xi^2\rangle^{2-2c}}d\tau\leq c$. This completes the estimate for $w_3$.
	
The statement (c) is a standard estimate and can be found in \cite{GTV}. 
	\end{proof}Let
\begin{equation*}
\Lambda(u)(t)=e^{-it\partial_x^2}\tilde{u}_0+\mathcal{D}(N(u,\overline{u}))(t)+\mathcal{L}h(t),
\end{equation*}
where, $h(t)=\chi_{(0,+\infty)}\big(f(t)-e^{-it\partial_x^2}\tilde{u}_0\big|_{x=0}-\mathcal{D}(N(u,\overline{u}))(t)\big|_{x=0}\big)\big|_{(0,+\infty)}$ and $\tilde{u}_0$ is a good extension of $u_0$ in  $\mathbb{R}$.

Using \eqref{forçante} and \eqref{nonlinear}, we see that 
\begin{equation*}
(i\partial_t-\partial_x^2)\Lambda(u)(t)=N(u,\overline{u})+2e^{i\frac{\pi}{4}}\delta_0(x)\mathcal{I}_{-\frac{1}{2}}h(t)
\end{equation*}
and $\Lambda(u)(0,t)=f(t)$.

Thus, a way to solve the IBVP \eqref{SK} is to prove that $\Lambda$ (or more precisely its time truncated versions) defines a contraction in a suitable Banach space. However, as pointed out before, since we are interested in low regularity results, we will use the auxiliary Bourgain spaces. Because of the estimate for the Duhamel forcing operator in our study, we need to take $b<\frac{1}{2}$, in order to Lemma \ref{edbf} (c) to be valid. However, the space $X^{s,b}$, with $b<\frac{1}{2}$, fails to be embedded into $C\big(\R_t;\,H^s(\mathbb{R}_x)\big)$. For this reason, we choose to work with the Banach space $Z$ given by 
\begin{equation*}
Z=C\big(\mathbb{R}_t;\,H^s(\mathbb{R}_x)\big)\cap C\big(\mathbb{R}_x;\,H^{\frac{2s+1}{4}}(\mathbb{R}_t)\big)\cap X^{s,b}.
\end{equation*}
In Section \ref{bilinear} we will obtain the bilinear estimates associated to nonlinearities $N_1(u,\overline{u})=u^2$ and $N_3(u,\overline{u})=\overline{u}^2$, in Bourgain spaces $X^{s,b}$, for $-\frac{3}{4}<s\leq 0$.

Unfortunately, we will able to prove the estimate for the Duhamel boundary forcing operator in Bourgain spaces 
$$\|\psi(t)\mathcal{L}f(x,t)\|_{X^{s,b}}\leq \|f\|_{H_0^\frac{2s+1}{4}(\mathbb{R}^+)},$$
if $-\frac{1}{2}<s<\frac{1}{2}$. In the next section, inspired by \cite{Holmerkdv}, to obtain this estimate on the full interval $(-\frac{3}{4},0]$, we will define an analytic family of boundary operators $\mathcal{L}^{\lambda}$, for $\lambda \in \mathbb{C}$, such that $\mathcal{L}^{0}=\mathcal{L}$ and
\begin{equation*}
\begin{cases}
(i\partial_t-\partial_x^2)\mathcal{L}^{\lambda}f(x,t)=2e^{i\frac{\pi}{4}}\frac{x_{-}^{\lambda-1}}{\Gamma(\lambda)}\mathcal{I}_{-\frac{1}{2}-\frac{\lambda}{2}}f(t),& (x,t)\in\mathbb{R}^2,\\
\mathcal{L}^{\lambda}f(0,t)=e^{i\frac{3\lambda\pi}{4}} f(t).& t\in\mathbb{R}.
\end{cases}
\end{equation*}
Due to the support properties of $\frac{x_{-}^{\lambda-1}}{\Gamma(\lambda)}$, $(i\partial_t-\partial_x^2)\mathcal{L}^{\lambda}f(x,t)=0$ for $x>0$, in the sense of distributions.  For any $-\frac{3}{4}<s\leq 0$, we will be able to address the right half-line problem \eqref{SK} by replacing $\mathcal{L}$ to $\mathcal{L}^{\lambda}$ for suitable $\lambda=\lambda(s)$.
\section{The Duhamel Boundary Forcing Operator Class}\label{section6}
Define, for $\lambda\in \mathbb{C}$, such that Re $\lambda>-2$,  and $f\in C_0^{\infty}(\mathbb{R}^+)$,
\begin{equation*}
\mathcal{L}^{\lambda}f(x,t)=\left[\frac{x_{-}^{\lambda-1}}{\Gamma(\lambda)}*\mathcal{L}\big(\mathcal{I}_{-\frac{\lambda}{2}}f\big)(\cdot,t)   \right](x).
\end{equation*}

It follows that
\begin{equation}\label{classe1}
\mathcal{L}^{\lambda}f(x,t)=\frac{1}{\Gamma(\lambda)}\int_{x}^{+\infty}(y-x)^{\lambda-1}\mathcal{L}\big(\mathcal{I}_{-\frac{\lambda}{2}}f\big)(y,t)dy,\; \; \text{for Re}\; \lambda>0.
\end{equation}

For Re $\lambda>-2$, by using    
\eqref{forçante00}, we see that
\begin{equation}\label{defnova}
\begin{split}
\mathcal{L}^{\lambda}f(x,t)&=\frac{1}{\Gamma(\lambda+2)}\int_{x}^{+\infty}(y-x)^{\lambda+1}\partial_y^2\mathcal{L}\big(\mathcal{I}_{-\frac{\lambda}{2}}f\big)(y,t)dy\\
&=\frac{1}{\Gamma(\lambda+2)}\int_{x}^{+\infty}(y-x)^{\lambda+1}\big(i\partial_{t}\mathcal{I}_{-\frac{\lambda}{2}}f\big)(y,t)dy-\frac{2e^{i\frac{\pi}{4}}x_{-}^{\lambda+1}}{\Gamma(\lambda+2)}\mathcal{I}_{-1/2-\lambda/2}f(t).
\end{split}
\end{equation}

From \eqref{forçante00} we see that 
\begin{equation*}
(i\partial_t-\partial_x^2)\mathcal{L}^\lambda f(x,t)=2e^{\frac{i\pi}{4}}\frac{x_{-}^{\lambda-1}}{\Gamma(\lambda)}\mathcal{I}_{-\frac{1}{2}-\frac{\lambda}{2}}f(t),
\end{equation*}
in the sense of distributions.

Using Lemma \ref{continuidade}, we see that $\mathcal{L}^{\lambda}f(x,t)$ is well defined for Re $ \lambda>-2$ and $t\in [0,1]$.

As $\frac{x_{-}^{\lambda-1}}{\Gamma(\lambda)}\bigg|_{\lambda=0}=\delta_0$, then $\mathcal{L}^{0}f(x,t)=\mathcal{L}f(x,t)$ and \eqref{defnova} implies $\mathcal{L}^{-1}f(x,t)=\partial_x\mathcal{L}(\mathcal{I}_{1/2})f(x,t)$.

The dominated convergence Theorem and Lemma \ref{continuidade} imply that, for fixed $t\in [0,1]$ and Re $ \lambda>-1$, $\mathcal{L}^{\lambda}f(x,t)$ is continuous in $x$ for $x\in \mathbb{R}$. 

The next Lemma states the values of $\mathcal{L}^\lambda f(x,t)$ at $x=0$.
\begin{lemma}
Let $f\in C_0^{\infty}(\R^+)$. If Re $\lambda>-1$, then
	\begin{equation}\label{lzero}
	\mathcal{L}^{\lambda}f(0,t)=e^{i\frac{3\lambda\pi}{4}}f(t).
	\end{equation}
\end{lemma}
\begin{proof} From (\ref{defnova}), we have
	\begin{equation}\label{cont.anal.}
	\mathcal{L}^{\lambda}f(0,t)=\int_{0}^{+\infty}\frac{y^{\lambda+1}}{\Gamma(\lambda+2)}\partial_y^2\mathcal{L}\big(\mathcal{I}_{-\frac{\lambda}{2}}f\big)(y,t)dy.
	\end{equation}
	
	By complex differentiation under the integral sign, \eqref{cont.anal.} proves that $\mathcal{L}^{\lambda}f(0,t)$ is analytic in $\lambda$, for Re $\lambda>-1$. By analyticity, we shall only compute \eqref{lzero} for $0<\lambda<2$.
	
	For the computation in the range $0<\lambda<2$, we use the representation \eqref{classe1}, to obtain
	\begin{equation*}
	\begin{split}
		\mathcal{L}^{\lambda}f(0,t)&=\frac{1}{\Gamma(\lambda)}\int_{0}^{+\infty}(y)^{\lambda-1}\mathcal{L}\big(\mathcal{I}_{-\frac{\lambda}{2}}f\big)(y,t)dy\\
		&=\frac{1}{\Gamma(\lambda)}\int_0^t(t-t')^{-\frac{1}{2}}\mathcal{I}_{-\frac{1}{2}-\frac{\lambda}{2}}f(t')\int_{0}^{+\infty}y^{\lambda-1}\frac{1}{\sqrt{\pi}}e^{\frac{-iy^2}{4(t-t')}}dydt'.
		\end{split}
	\end{equation*}
		Set $I=\int_{0}^{+\infty}y^{\lambda-1}e^{\frac{-iy^2}{4(t-t')}}dy$. Changing variables $r=\frac{y^2}{4(t-t')}$, then $y=r^{\frac{1}{2}}2(t-t')^{\frac{1}{2}}$ and $dy=r^{-\frac{1}{2}}(t-t')^{\frac{1}{2}}dr$, it follows that
			\begin{equation*}
		I=\int_0^{+\infty}[r^{\frac{1}{2}}2(t-t')^{\frac{1}{2}}]^{\lambda-1}r^{-\frac{1}{2}}e^{-ir}(t-t')^{\frac{1}{2}}dr
		=2^{\lambda-1}(t-t')^{\frac{\lambda}{2}}\int_0^{+\infty}r^{\frac{\lambda}{2}-1}e^{-ir}dr.
	\end{equation*}
	
	By a change of contour, 
\begin{equation*}
	I=2^{\lambda-1}(t-t')^{\frac{\lambda}{2}}i^{\frac{3\lambda}{2}}\int_0^{+\infty}r^{\frac{\lambda}{2}-1}e^{-r}dr=2^{\lambda-1}(t-t')^{\frac{\lambda}{2}}i^{\frac{3\lambda}{2}}\Gamma\left(\frac{\lambda}{2}\right),\ \text{for}\ \lambda\in (0,2).
	\end{equation*}
	Using the formula $\frac{\Gamma(\frac{\lambda}{2})}{\Gamma(\lambda)}=\frac{2^{1-\lambda}\sqrt{\pi}}{\Gamma(\frac{\lambda}{2}+\frac{1}{2})},$ for $\lambda\in \mathbb{R}^+$, we obtain	
\begin{eqnarray*}
		\mathcal{L}^{\lambda}f(0,t)&=&\frac{2^{\lambda-1}}{\sqrt{\pi}}\frac{\Gamma\left(\frac{\lambda}{2}\right)}{\Gamma(\lambda)}i^{\frac{3\lambda}{2}}\int_0^t(t-t')^{\frac{\lambda}{2}-\frac{1}{2}}\mathcal{I}_{-\frac{\lambda}{2}-\frac{1}{2}}f(t')dt'\\
		&=&\frac{1}{\Gamma\left(\frac{\lambda}{2}+\frac{1}{2}\right)}i^{\frac{3\lambda}{2}}\int_0^t(t-t')^{\frac{\lambda}{2}-\frac{1}{2}}\mathcal{I}_{-\frac{\lambda}{2}-\frac{1}{2}}f(t')dt'\\
		&=&i^{\frac{3\lambda}{2}}\mathcal{I}_{\frac{\lambda}{2}+\frac{1}{2}}\mathcal{I}_{-\frac{\lambda}{2}-\frac{1}{2}}f(t)=e^{i\frac{3\lambda\pi}{4}}f(t).
			\end{eqnarray*}\end{proof}
Now we state the needed estimates for the Duhamel boundary forcing operators class.
\begin{lemma}\label{edbf}
	\begin{itemize}Let $s\in\mathbb{R}$. Then
		\item[(a)](Space traces)  If $s-\frac{3}{2}<\lambda<s+\frac{1}{2},\ \lambda>-1$ and supp$\ f\subset [0,1]$, then
	$$\|\mathcal{L}^{\lambda}f(x,t)\|_{C\big(\mathbb{R}_t;\,H^s(\mathbb{R}_x)\big)}\leq c \|f\|_{H_0^\frac{2s+1}{4}(\mathbb{R}^+)};$$
		\item[(b)](Time traces) If $-1<\lambda<1$, then 
		$$\|\psi(t)\mathcal{L}^{\lambda}f(x,t)\|_{C\big(\mathbb{R}_x;\,H_0^{\frac{2s+1}{4}}(\mathbb{R}_t^+)\big)}\leq c \|f\|_{H_0^\frac{2s+1}{4}(\mathbb{R}^+)};$$
		
		\item[(c)](Bourgain spaces) If $b<\frac{1}{2}$, $s-\frac{1}{2}<\lambda<s+\frac{1}{2}$ and $-1<\lambda<\frac{1}{2}$, then
		$$\|\psi(t)\mathcal{L}^{\lambda}f(x,t)\|_{X^{s,b}}\leq c \|f\|_{H_0^\frac{2s+1}{4}(\mathbb{R}^+)}.$$
	\end{itemize}
\end{lemma}
\begin{remark}
	As in the treatment for the IBVP associated to KdV equation \cite{CK}, the assumption $b<\frac{1}{2}$ is crucial in the proof of Lemma \ref{edbf} (c). This fact forced us to work with the Bourgain spaces with $b<\frac{1}{2}$. In \cite{tzirakis} was derived regularity properties for the cubic nonlinear Scr\"odinger equation on the half-line by using Laplace transform method, being  also necessary to work with Bourgain spaces $X^{s,b}$ with $b<1/2$. 
\end{remark}
\begin{proof}
We follow closely the argument
in \cite{Holmerkdv}. By density, we can assume that $f\in C_{0,c}^{\infty}(\mathbb{R}^{+})$. Using $\mathcal{F}_x\left(\frac{x_{-}^{\lambda-1}}{\Gamma(\lambda)}\right)(\xi)=c(\xi-i0)^{-\lambda}$, we see that $\mathcal{F}_x(\mathcal{L}^{\lambda}f)(\xi,t)=c(\xi-i0)^{-\lambda}\int_0^te^{i(t-t')\xi^2}\mathcal{I}_{-\frac{\lambda}{2}-\frac{1}{2}}f(t')dt'$. As $\lambda>-1$, the distribution $(\xi-i0)^{-\lambda}$ is a locally integrable function. The changing of variables $\eta=\xi^2$ gives that, for fixed $t$,
\begin{eqnarray*}
		\|\mathcal{L}^{\lambda}f(x,t)\|_{H^s(\mathbb{R})}^2&\leq& c \int_{\eta}|\eta|^{-\lambda-\frac{1}{2}}\langle\eta\rangle^s\left|\int_0^te^{i(t-t')\eta}\mathcal{I}_{-\frac{\lambda}{2}-\frac{1}{2}}f(t')dt'\right|^2d\eta\\
		&=&c\int_{\eta}|\eta|^{-\lambda-\frac{1}{2}}\langle\eta\rangle^s\left|\big(\chi_{(-\infty,t)}\mathcal{I}_{-\frac{\lambda}{2}-\frac{1}{2}}f\big)^{\widehat{}}(\eta)\right|^2d\eta.
	\end{eqnarray*}

	Since $-1<s-\lambda-\frac{1}{2}<1$, Lemma \ref{sobolev0} implies
	\begin{equation*}
	\int_{\eta}|\eta|^{-\lambda-\frac{1}{2}}\langle\eta\rangle^s\left|(\chi_{(-\infty,t)}\mathcal{I}_{-\frac{\lambda}{2}-\frac{1}{2}}f)^{\widehat{}}(\eta)\right|^2d\eta\leq c\int_{\eta}\langle\eta\rangle^{s-\lambda-\frac{1}{2}}\left|(\chi_{(-\infty,t)}\mathcal{I}_{-\frac{\lambda}{2}-\frac{1}{2}}f)^{\widehat{}}(\eta)\right|^2d\eta.
	\end{equation*}
	By Lemmas \ref{sobolevh0} (to remove $\chi_{(-\infty,t)})$ and \ref{lio} (to estimate $\mathcal{I}_{-\frac{\lambda}{2}-\frac{1}{2}}$), we obtain
	\begin{equation*}
	c\int_{\eta}\langle\eta\rangle^{s-\lambda-\frac{1}{2}}\left|(\chi_{(-\infty,t)}\mathcal{I}_{-\frac{\lambda}{2}-\frac{1}{2}}f)^{\widehat{}}(\eta)\right|^2d\eta\leq c \|f\|_{H_0^{\frac{2s+1}{4}}(\mathbb{R}^+)}^2,	\end{equation*}
	which proves (a). Now we prove (b). By Lemma \ref{cut}, we can ignore the test function.
	
	Changing variables $t\rightarrow  t-t'$, we get
	\begin{eqnarray*}
	& &(I-\partial_t^2)^{\frac{2s+1}{4}}\left(\frac{x_{-}^{\lambda-1}}{\Gamma(\lambda)}*\int_{-\infty}^te^{-i(t-t')\partial_x^2}\delta(x)h(t')dt'\right)\\      
	& &\quad=\left(\frac{x_{-}^{\lambda-1}}{\Gamma(\lambda)}*\int_{-\infty}^te^{-i(t-t')\partial_x^2}\delta(x)(I-\partial_{t'}^2)^{\frac{2s+1}{4}}h(t')dt'\right).
	\end{eqnarray*}
	It suffices to prove
	
		\begin{equation*}
	\left\|\int_{\xi}e^{ix\xi}(\xi-i0)^{-\lambda}\int_{-\infty}^te^{i(t-t')\xi^2}(\mathcal{I}_{-\frac{\lambda}{2}-\frac{1}{2}}f)(t')dt'd\xi\right\|_{L_x^{\infty}L_t^2(\mathbb{R})} \leq c \|f\|_{L_t^2(\mathbb{R}^+)}.
	\end{equation*}
	Using $\chi_{(-\infty,t)}=\frac{1}{2}sgn(t-t')+\frac{1}{2}$, we obtain
		\begin{eqnarray*}
		& &\int_{\xi}e^{ix\xi}(\xi-i0)^{-\lambda}\int_{-\infty}^te^{i(t-t')\xi^2}(\mathcal{I}_{-\frac{\lambda}{2}-\frac{1}{2}}f)(t')dt'd\xi\\
		& &\quad =\int_{\xi}e^{ix\xi}(\xi-i0)^{-\lambda}\int_{-\infty}^{+\infty}\chi_{(-\infty,t)}e^{i(t-t')\xi^2}(\mathcal{I}_{-\frac{\lambda}{2}-\frac{1}{2}}f)(t')dt'd\xi\\
		& &\quad \quad +\int_{\xi}e^{ix\xi}(\xi-i0)^{-\lambda}\int_{-\infty}^{+\infty}\frac{1}{2}e^{i(t-t')\xi^2}(\mathcal{I}_{-\frac{\lambda}{2}-\frac{1}{2}}f)(t')dt'd\xi:= I+II.
	\end{eqnarray*}
	
	We treat $I$ and $II$ separately. The change of variables $\eta=\xi^2$ implies
	\begin{eqnarray*}
		II&=&\int_{\xi}e^{ix\xi}(\mathcal{I}_{-\frac{\lambda}{2}-\frac{1}{2}}f)^{\widehat{}}(\xi^2)(\xi-i0)^{-\lambda}e^{it\xi^2}d\xi=\int_{\xi}e^{ix\xi}(\xi^2-i0)^{\frac{\lambda+1}{2}}\hat{f}(\xi^2)(\xi-i0)^{-\lambda}e^{it\xi^2}d\xi\\
		&=&\int_{\eta=0}^{+\infty}e^{ix\eta^{\frac{1}{2}}}(\eta-i0)^{\frac{\lambda+1}{2}}\hat{f}(\eta)(\eta^{\frac{1}{2}}-i0)^{-\lambda}e^{it\eta}\eta^{-\frac{1}{2}} d\eta\\
		& &+\int_{\eta=0}^{+\infty}e^{-ix\eta^{\frac{1}{2}}}(\eta-i0)^{\frac{\lambda+1}{2}}\hat{f}(\eta)(-\eta^{\frac{1}{2}}-i0)^{-\lambda}e^{it\eta}\eta^{-\frac{1}{2}} d\eta\\
		&=&c_1\int_{0}^{+\infty}e^{ix\eta^{\frac{1}{2}}}\hat{f}(\eta)e^{it\eta}d\eta +c_2e^{-ix\eta^{\frac{1}{2}}}\hat{f}(\eta)e^{it\eta}d\eta,
			\end{eqnarray*}
		which implies $\|II(x,t)\|_{L_t^2}\leq c\|f\|_{L_t^2}$.
	To estimate $I$, we write $I=\left(\frac{1}{2}e^{i\xi^2\cdot}sgn(\cdot)*\mathcal{I}_{-\frac{\lambda}{2}-\frac{1}{2}}f\right)(t).$
	Since $$\mathcal{F}_t\left(\frac{1}{2}e^{i\xi^2\cdot}sgn(\cdot)*\mathcal{I}_{-\frac{\lambda}{2}-\frac{1}{2}}f\right)(\tau)=\frac{(\tau-i0)^{\frac{1+\lambda}{2}}\hat{f}(\tau)}{\tau-\xi^2},$$ it follows that
	\begin{equation*}
	I=\int_{\tau}e^{it\tau}\lim_{\epsilon\rightarrow 0}\int_{|\tau-\xi^2|>\epsilon}\frac{e^{ix\xi}(\tau-i0)^{\frac{\lambda+1}{2}}(\xi-i0)^{-\lambda}d\xi}{(\tau-\xi^2)}\hat{f}(\tau)d\tau.
	\end{equation*}
	Therefore, it suffices to show that the function $$g(\tau):=\lim_{\epsilon\rightarrow 0}\int_{|\tau-\xi^2|>\epsilon}\frac{e^{ix\xi}(\tau-i0)^{\frac{\lambda+1}{2}}(\xi-i0)^{-\lambda}d\xi}{(\tau-\xi^2)}$$is limited. 
	
	Changing of variables $\xi\rightarrow |\tau|^{\frac{1}{2}}\xi$ and using $
	(\xi|\tau|^{\frac{1}{2}}-i0)^{-\lambda}=|\tau|^{-\frac{\lambda}{2}}(c_1\xi_{+}^{-\lambda}+c_2\xi_{-}^{-\lambda}),
	$
	we obtain
	
	\begin{eqnarray*}
		g(\tau)&=&\int \frac{e^{i|\tau|^{\frac{1}{2}}x\xi}|\tau|^{\frac{1-\lambda}{2}}(c_1\tau_{+}^{\frac{\lambda+1}{2}}+c_2\tau_{-}^{\frac{\lambda+1}{2}})(\xi-i0)^{-\lambda}d\xi}{\tau-|\tau|\xi^2}\\
		&=&c_1\chi_{\{\tau>0\}}\int_{\xi}e^{i|\tau|^{\frac{1}{2}}x\xi}\frac{c_1\xi_{+}^{-\lambda}+c_2\xi_{-}^{-\lambda}}{1-\xi^2}d\xi + c_2\chi_{\{\tau<0\}}\int_{\xi}e^{i\tau^{\frac{1}{2}}x\xi}\frac{c_1\xi_{+}^{-\lambda}+c_2\xi_{-}^{-\lambda}}{1+\xi^2}d\xi\\
		&:=&g_1+g_2.
	\end{eqnarray*}
The second integral is uniformly limited in $\tau$ if $\lambda<1$, this can be obtained by considering the cases $|\xi|<1$ and $|\xi|\geq 1$.
	Now we estimate $g_1$. Let $\theta\in C^{\infty}(\mathbb{R})$ such that $\theta(\xi)=1$ in $[\frac{3}{4},\frac{4}{3}]$ and $\theta(t)=0$ outside $(\frac{1}{2},\frac{3}{2})$. Then we obtain
		\begin{eqnarray*}
		g_1&=&c_1\chi_{\{\tau>0\}}\int_{\xi}e^{i|\tau|^{\frac{1}{2}}x\xi}\frac{c_1\xi_{+}^{-\lambda}\theta(\xi)}{1-\xi^2}d\xi+c_1\chi_{\{\tau>0\}}\int_{\xi}e^{i|\tau|^{\frac{1}{2}}x\xi}\frac{1-\theta(\xi)(c_1\xi_{+}^{-\lambda}+c_2\xi_{-}^{-\lambda})}{1-\xi^2}d\xi\\
		&=&g_{11}+g_{12}.
	\end{eqnarray*}
	The second integral is clearly limited when $\lambda>-1$. To estimate $g_{11}$, we write
\begin{eqnarray*}
		g_{11}&=&c_1\chi_{\{\tau>0\}}\int_{\xi}e^{i|\tau|^{\frac{1}{2}}x\xi}\frac{c_1\xi_{+}^{-\lambda}\theta(\xi)}{(1-\xi)(1+\xi)}d\xi=c_1\chi_{\{\tau>0\}}\mathcal{F}_{\xi}^{-1}\left( \frac{c_1\xi_{+}^{-\lambda}\theta(\xi)}{(1-\xi)(1+\xi)} \right)(x|\tau|^{\frac{1}{2}})\\
		&=&\left[(sgn(1-\xi)*\mathcal{F}_{\xi}^{-1}\left( \frac{c_1\xi_{+}^{-\lambda}\theta(\xi)}{(1+\xi)} \right)\right](x|\tau|^{\frac{1}{2}}).
	\end{eqnarray*}
	This becomes convolution of Schwartz class function with a phase shifted sgn $(x)$ function, completing the proof
of (b).
	
	To prove (c), we first note that
	\begin{equation*}
	\mathcal{F}_x(\mathcal{L}^{\lambda}f)(\xi,t)=c(\xi-i0)^{-\lambda}\int\frac{e^{it\tau}-e^{it\xi^2}}{\tau-\xi^2}(\tau-i0)^{\frac{\lambda}{2}+\frac{1}{2}}\hat{f}(\tau)d\tau.
	\end{equation*}
	Now let $\theta(\tau)\in C^{\infty}$ such that $\theta(\tau)=1$ for $|\tau|\leq 1$ and $\theta(\tau)=0$ for $|\tau|\geq 2$.
	Define $u_1$, $u_2$, $u_3$ by
	\begin{equation*}
	\mathcal{F}_xu_1(\xi,t)=(\xi-i0)^{-\lambda}\int\frac{e^{it\tau}-e^{it\xi^2}}{\tau-\xi^2}\theta(\tau-\xi^2)(\tau-i0)^{\frac{\lambda}{2}+\frac{1}{2}}\hat{f}(\tau)d\tau,
	\end{equation*} 
	\begin{equation*}
	\mathcal{F}_xu_2(\xi,t)=(\xi-i0)^{-\lambda}\int\frac{e^{it\tau}}{\tau-\xi^2}(1-\theta(\tau-\xi^2))(\tau-i0)^{\frac{\lambda}{2}+\frac{1}{2}}\hat{f}(\tau)d\tau,
	\end{equation*} 
		\begin{equation*}
	\mathcal{F}_x{u}_3(\xi,t)=(\xi-i0)^{-\lambda}\int\frac{e^{it\xi^2}}{\tau-\xi^2}(1-\theta(\tau-\xi^2))(\tau-i0)^{\frac{\lambda}{2}+\frac{1}{2}}\hat{f}(\tau)d\tau,
	\end{equation*} 
	it follows that $\mathcal{L}^{\lambda}f=u_1+u_{2}-u_{3}$.
	
	We start by estimating $u_2$. It's immediate that
		\begin{eqnarray*}
		\|u_2\|_{X^{s,b}}^2&\leq& c\int\int\langle \xi\rangle^{2s}\frac{\langle \tau-\xi^2\rangle^{2b}}{|\tau-\xi^2|^2}(1-\theta(\tau-\xi^2))^2|\xi|^{-2\lambda}|\tau|^{\lambda+1}|\hat{f}(\tau)|^2d\tau d\xi\\
		&=&c\int  |\tau|^{\lambda+1}\left( \int\frac{|\xi|^{-2\lambda}\langle \xi\rangle^{2s}d\xi}{\langle \tau-\xi^2\rangle^{2-2b}}\right)|\hat{f}(\tau)|^2d\tau.
	\end{eqnarray*}
	Thus, it suffices to obtain
	\begin{equation}\label{tna}
	I(\tau)=\int\frac{|\eta|^{-\frac{1}{2}-\lambda}\langle \eta\rangle^{s}d\eta}{\langle \tau-\eta\rangle^{2-2b}}\leq c \langle\tau\rangle^{s-\lambda-\frac{1}{2}}.
	\end{equation}
	This will be obtained by separating some cases. For $|\tau|\leq \frac{1}{2}$, we have that $\langle\tau-\eta\rangle\sim\langle\eta\rangle$. It follows that
	
\begin{equation*}
I(\tau)\leq c \int\frac{|\eta|^{-\frac{1}{2}-\lambda}\langle \eta\rangle^{s}d\eta}{\langle \tau-\eta\rangle^{2-2b}}\leq  c \int_{|\eta|\leq 1} |\eta|^{-\frac{1}{2}-\lambda}+c\int\frac{d\eta}{\langle \eta\rangle^{2-2b-s+\frac{1}{2}+\lambda}}\leq c,
\end{equation*}
where we have used that $\lambda<\frac{1}{2}$ and $\lambda>-\frac{3}{2}+2b+s$.
	
	Now suppose that $|\tau|>\frac{1}{2}$ and $|\eta|\geq \frac{|\tau|}{2}$. In this case we use $\lambda\geq s-\frac{1}{2}$ and $b<\frac{1}{2}$ to obtain 
	\begin{equation*}
	I(\tau)\leq c \int\frac{\langle\eta\rangle^{s-\frac{1}{2}-\lambda}}{\langle\tau-\eta\rangle^{2-2b}}\leq c \langle\tau\rangle^{s-\frac{1}{2}-\lambda}\int\frac{d\eta}{\langle\tau-\eta\rangle^{2-2b}}\leq c \langle\tau\rangle^{s-\frac{1}{2}-\lambda}.
	\end{equation*}
		On the case $|\tau|\geq \frac{1}{2}$ and $|\eta|\leq \frac{|\tau|}{2}$ we have $\langle \tau-\eta\rangle\geq \frac{1}{2}\langle \tau\rangle$. Then
		\begin{equation*}
	I(\tau)\leq c\langle\tau\rangle^{2b-2}\int_{|\eta|\leq \frac{|\tau|}{2}}\langle\eta\rangle^s|\eta|^{-\frac{1}{2}-\lambda}\leq c\langle\tau\rangle^{2b-2}\langle\tau\rangle^{s+\frac{1}{2}-\lambda}\leq c\langle\tau\rangle^{s-\frac{1}{2}-\lambda}.
	\end{equation*}
	This completes the estimate for $u_2$. To estimate $u_3$, we write
	 $u_3(x,t)=\theta(t)e^{-it\partial_x^2}\phi(x)$, where   
	  \begin{equation*}
	\hat{\phi}(\xi)=(\xi-i0)^{-\lambda}\int \frac{1-\psi(\tau-\xi^2)}{\tau-\xi^2}\big(\mathcal{I}_{-\frac{\lambda}{2}-\frac{1}{2}}f\big)^{\widehat{}}(\tau)d\tau.
	\end{equation*}
		Then, by Lemma \ref{grupo}, it suffices to show that
		\begin{equation}\label{23110}
	\|\phi\|_{H^s(\R)}\leq c \|f\|_{H_0^{\frac{2s+1}{4}}(\R^+)}.
	\end{equation}
	Arguing as in the proof of Lemma 5.8 in \cite{Holmerkdv}, there exists a function $\beta\in S(\mathbb{R})$ such that
	\begin{equation*}
	\frac{1-\psi(\tau-\xi^2)}{\tau-\xi^2}\big(\mathcal{I}_{-\frac{\lambda}{2}-\frac{1}{2}}f\big)^{\widehat{}}(\tau)d\tau=\int_{\tau}(\mathcal{I}_{-\frac{\lambda}{2}-\frac{1}{2}}f)^{\widehat{}}(\tau)\beta(\tau-\xi^2)d\tau.
	\end{equation*} 
	By using Cauchy-Schwarz and $|\beta(\tau-\xi^2)|\leq c \langle\tau-\xi^2\rangle^{-N}$, for $N>>1$, we have that
	\begin{eqnarray}
		\|\phi\|_{H^s(\R)}^2&\leq& \int_{\xi}\langle\xi\rangle^{2s}|\xi|^{-2\lambda}\left(\int_{\tau}\beta(\tau-\xi^2)|\tau|^{\frac{\lambda+1}{2}}\hat{f}(\tau)d\tau\right)^2d\xi\nonumber\\
		&\leq& c\int_{\tau}\left(\int_{\eta}|\eta|^{-\lambda-\frac{1}{2}}\langle\eta\rangle^{s}\langle\tau-\eta\rangle^{-2N+2}d\eta\right)|\tau|^{\lambda+1}|\hat{f}(\tau)|^2d\tau.\label{23111}
	\end{eqnarray}
		Using (\ref{tna}), we see that
			\begin{equation}\label{23112}
	\int_{\eta}|\eta|^{-\lambda-\frac{1}{2}}\langle\eta\rangle^{s}\langle\tau-\eta\rangle^{-2N+2}d\eta\leq c\langle\tau\rangle^{s-\lambda-\frac{1}{2}}.
	\end{equation}
	Substituting \eqref{23112} in \eqref{23111} we obtain \eqref{23110}.
	
	Finally we estimate $u_1$. By the power series expansion for $e^{it(\tau-\xi^2)}$, we write $u_1(x,t)=\displaystyle\sum_{n=1}^{\infty}\frac{\psi_k(t)}{k!}e^{-t\partial_x^2}\phi_k(x),$ where $\psi_k(t)=i^kt^k\psi(t)\ \text{and}\  $$\psi_k(t)=i^kt^k\psi(t)$ and $$\hat{\phi}_k(\xi)=(\xi-i0)^{-\lambda}\int_{\tau}(\tau-\xi^2)^{k-1}\theta(\tau-\xi^2)(\tau)^{\frac{1}{2}+\frac{\lambda}{2}}\hat{f}(\tau)d\tau.$$ Using \eqref{grupo}, we need to show that $\|\phi_k\|_{H^{s}(\R)}\leq c \|f\|_{H_0^{\frac{2s+1}{4}}(\R^+)}$.
	
	 By Cauchy-Schwarz inequality,  
\begin{equation}\label{crb100}
		\|\phi_k\|_{H^{s}(\R)}^2\leq c\int_{\xi}\langle \xi\rangle^{2s}|\xi|^{-2\lambda}\int_{|\tau-\xi^2|\leq 1}|\tau|^{\lambda+1}|\hat{f}(\tau)|^2d\tau d\xi.
	\end{equation}
	The changing of variables $\eta=\xi^2$ implies 
\begin{equation}\label{crb110}
	\int_{|\tau-\xi^2|\leq 1}\langle \xi\rangle^{2s}|\xi|^{-2\lambda}d\xi\leq c\int_{|\tau-\eta|\leq 1}|\eta|^{-\frac{1}{2}-\lambda}\langle \eta\rangle^{s}d\eta\leq c\langle \tau \rangle^{s-\lambda-\frac{1}{2}}.
	\end{equation}
	
		Combining (\ref{crb100}) and (\ref{crb110}), we obtain  $\|\phi_k\|_{H^{s}(\R)}\leq c \|f\|_{H_0^{\frac{2s+1}{4}}(\R^+)}$. \end{proof}

\section{Bilinear Estimates}\label{bilinear}
\begin{lemma}\label{estimativan1}(Estimates for nonlinear term $N_1$)
	\begin{itemize}
		\item[(a)] Given $-\frac{3}{4}<s\leq 0$, there exists  $b=b(s)<\frac{1}{2}$ such that
		\begin{equation}\label{objetivo}
		\|uv\|_{X^{s,-b}}\leq c\|u\|_{X^{s,b}}\|v\|_{X^{s,b}}.
		\end{equation}
		
		\item[(b)] Given $-\frac{3}{4}<s\leq -\frac{1}{2}$, there exists $b=b(s)<\frac{1}{2}$ such that
		\begin{equation}\label{objetivoabc}
		\|uv\|_{W^{s,-b}}\leq c\|u\|_{X^{s,b}}\|v\|_{X^{s,b}}.
		\end{equation}
			\end{itemize}	
	\end{lemma}
\begin{lemma}\label{estimativan2}(Estimate for nonlinear term $N_2$)
	 Given $-\frac{1}{4}<s\leq 0$, there exists $b=b(s)<\frac{1}{2}$ such that
		\begin{equation*}
		\|u\overline{v}\|_{X^{s,-b}}\leq c\|u\|_{X^{s,b}}\|v\|_{X^{s,b}}.
		\end{equation*}
\end{lemma}
\begin{lemma}\label{estimativan3}(Estimates for nonlinear term $N_3$)
	\begin{itemize}
		\item[(a)] Given $-\frac{3}{4}<s\leq 0$, there exists $b=b(s)<\frac{1}{2}$ such that
		\begin{equation*}
		\|\overline{u}\overline{v}\|_{X^{s,-b}}\leq c\|u\|_{X^{s,b}}\|v\|_{X^{s,b}}.
		\end{equation*}
		
		\item[(b)] Given $-\frac{3}{4}<s\leq -\frac{1}{2}$, there exists $b=b(s)<\frac{1}{2}$ such that
		\begin{equation*}
		\|\overline{uv}\|_{W^{s,-b}}\leq c\|u\|_{X^{s,b}}\|v\|_{X^{s,b}}.
		\end{equation*}
	\end{itemize}	
\end{lemma}	
We shall prove these lemmas by the calculus techniques of \cite{KPV2-a}. We begin with some elementary integral estimates, whose proofs can be found in \cite{GTV}, \cite{bop} and \cite{Holmerkdv}, respectively.
\begin{lemma}\label{lemagtv}
	Let  $b_1,b_2$, such that $b_1+b_2>\frac{1}{2}$ and $b_1,b_2< \frac{1}{2}$. Then
	\begin{equation*}
	\int\frac{dy}{\langle y-\alpha\rangle^{2b_1}\langle y-\beta\rangle^{2b_2}}\leq \frac{c}{\langle \alpha-\beta\rangle^{2b_1+2b_2-1}}.
	\end{equation*}
\end{lemma}
\begin{lemma}\label{lemanovo}
	If $b>\frac{1}{2}$, then
	\begin{equation*}
	\int_{-\infty}^{\infty}\frac{dx}{\langle\alpha_0+\alpha_1x+x^2\rangle^b}\leq c.
	\end{equation*}
\end{lemma}
\begin{lemma}\label{Hol}
	If $b<\frac{1}{2}$, then
	\begin{equation*}
	\int_{|x|<\beta}\frac{dx}{\langle x\rangle^{4b-1}|\alpha-x|^{\frac{1}{2}}}\leq c\frac{(1+\beta)^{2-4b}}{\langle\alpha\rangle^{\frac{1}{2}}}.
	\end{equation*}
\end{lemma}

\subsection{Proof of Lemma \ref{estimativan1}}
Let $\rho=-s\in [0,\frac{3}{4})$. For $u,v\in X^{s,b}=X^{-\rho,b}$ define 
$$g(\xi,\tau)=\langle\tau-\xi^2\rangle^b\langle\xi\rangle^{-\rho}\hat{u}(\xi,\tau),\ h(\xi,\tau)=\langle\tau-\xi^2\rangle^b\langle\xi\rangle^{-\rho}\hat{v}(\xi,\tau).$$
Then $\|g\|_{L^2}=\|u\|_{X^{s,b}}$ and $\|h\|_{L^2}=\|v\|_{X^{s,b}}$. We write the left hand side of (\ref{objetivo}) in terms of $f$ and $g$, i.e.,
\begin{equation*}
	\|uv\|_{X^{s,-b}}
	:=\sup_{\|\phi\|_{L^2}\leq 1}W(g,h,\phi),
\end{equation*}
where,
\begin{equation*}
	W(g,h,\phi)=\int_{\mathbb{R}^4}\frac{\langle \xi\rangle^s}{\langle\tau-\xi^2\rangle^{b}}\frac{g(\xi_1,\tau_1)}{\langle \xi_1\rangle^{s}\langle\tau_1-\xi_1^2\rangle^b}\frac{h(\xi-\xi_1,\tau-\tau_1)}{\langle \xi-\xi_1\rangle^{s}\langle\tau-\tau_1-(\xi-\xi_1)^2\rangle^b}\phi(\xi,\tau)d\xi_1d\tau_1d\xi d\tau.
\end{equation*}
Initially, we treat the case $\rho=0$. We integrate over $\xi_1$ and $\tau_1$ first, and then we use the Cauchy-Schwarz and H\"older  inequalities to obtain

\begin{equation}\label{ww1}
\begin{split}
W^2(g,h,\phi) &\leq  \|\phi\|_{L^2}^2 \times \left \|\frac{1}{\langle\tau-\xi^2\rangle^{b}}\int \int\frac{g(\xi_1,\tau_1) }{\langle\tau_1-\xi_1^2\rangle^b}\frac{h(\xi-\xi_1,\tau-\tau_1)}{\langle\tau-\tau_1-(\xi-\xi_1)^2\rangle^b}d\xi_1d\tau_1\right\|_{L_{\tau}^2L_{\xi}^2}^2\\
&\leq \|\phi\|_{L^2}^2\|g\|_{L^2}^2\|h\|_{L^2}^2\\
& \quad \times \left\| \frac{1}{\langle\tau-\xi^2\rangle^{2b}}\int \int\frac{d\tau_1 d\xi_1}{\langle\tau_1-\xi_1^2\rangle^{2b}\langle\tau-\tau_1-(\xi-\xi_1)^2\rangle^{2b} }\right\|_{L^{\infty}_{\xi,\tau}}.
\end{split}
\end{equation}
Using Lemmas \ref{lemagtv} and \ref{lemanovo}, we obtain
\begin{equation}\label{01121}
\frac{1}{\langle\tau-\xi^2\rangle^{2b}}\int \int\frac{d\tau_1 d\xi_1}{\langle\tau_1-\xi_1^2\rangle^{2b}\langle\tau-\tau_1-(\xi-\xi_1)^2\rangle^{2b} }\leq c\int_{-\infty}^{\infty}\frac{d\xi_1}{\langle\tau-\xi^2+2\xi_1(\xi-\xi_1)\rangle^{4b-1}},
\end{equation}
that is bounded provided $\frac{3}{8}<b<\frac{1}{2}$. This proves the case $\rho=0$.

To prove the case $\rho\in (1/2,3/4)$, we write
\begin{eqnarray*}
	W(g,h,\phi)&=&\int_{\mathbb{R}^4}\frac{\langle \xi\rangle^s}{\langle\tau-\xi^2\rangle^{b}}\frac{g(\xi_1,\tau_1)}{\langle \xi_1\rangle^{s}\langle\tau_1-\xi_1^2\rangle^b}\frac{h(\xi-\xi_1,\tau-\tau_1)}{\langle \xi-\xi_1\rangle^{s}\langle\tau-\tau_1-(\xi-\xi_1)^2\rangle^b}\phi(\xi,\tau)d\xi_1d\tau_1d\xi d\tau \nonumber \\
	&=&\int_{\mathcal{R}_1}+\int_{\mathcal{R}_2}+\int_{\mathcal{R}_3}+\int_{\mathcal{R}_4}:= W_1+W_2+W_3+W_4,\nonumber
\end{eqnarray*}
 where
\begin{equation*}
\begin{split}
\mathcal{R}_1&=\{(\xi,\xi_1,\tau,\tau_1)\in\mathbb{R}^4;\, |\xi_1|\leq 1 \},\\
\mathcal{R}_2&=\{(\xi,\xi_1,\tau,\tau_1)\in\mathbb{R}^4;\, |\xi_1|\geq 1,\, \max\{|\tau-\xi^2|,|\tau_1-\xi_1^2|, |\tau-\tau_1-(\xi-\xi_1^2)|\}=|\tau-\xi^2| \},\\
\mathcal{R}_3&=\{(\xi,\xi_1,\tau,\tau_1)\in\mathbb{R}^4;\, |\xi_1|\geq 1,\, \max\{|\tau-\xi^2|,|\tau_1-\xi_1^2|, |\tau-\tau_1-(\xi-\xi_1^2)|\}=|\tau_1-\xi_1^2| \},
\end{split}
\end{equation*}
\begin{equation*}
\begin{split}
\mathcal{R}_4&=\big\{(\xi,\xi_1,\tau,\tau_1)\in\mathbb{R}^4;\, |\xi_1|\geq 1, \max\{|\tau-\xi^2|,|\tau_1-\xi_1^2|, |\tau-\tau_1-(\xi-\xi_1^2)|\}\\
& \ \ \ \ \ \ \ \ \ \ \ \ \ \ \ \ \ \ \ \ \ \ \ \ \ \ \ \ \ \ \ \ \ \ \ \ \ \ \ \ \ \ \ \ \ \ \ \ \ \ \ \ \ \ \ \ \ \ \ \ \ \ \ \ \ \ \ \ \ \ \ \ \ \ \ \ \ \ \ \ \ \ \ \ \ \ \ =|\tau-\tau_1-(\xi-\xi_1^2)| \big\}.
\end{split}
\end{equation*}
The estimate for $W_1$ can be obtained as in the case $\rho=0$, since the frequencies cancel each other out. To estimate $W_2$ we integrate over $\xi_1$ and $\tau_1$ first, and then we use the Cauchy-Schwarz and H\"older  inequalities to obtain
\begin{eqnarray}
W_2^2 &\leq&c \|\phi\|_{L^2}^2\|g\|_{L^2}^2\|h\|_{L^2}^2\nonumber\\
      & &\quad \times \left\| \frac{|\xi|^{2s}}{\langle\tau-\xi^2\rangle^{2b}}\int \int\frac{\chi_{\{A(\xi,\tau)\}}d\tau_1 d\xi_1}{\langle \xi_1\rangle^{2s}\langle\tau_1-\xi_1^2\rangle^{2b}\langle\xi-\xi_1\rangle^{2s}\langle\tau-\tau_1-(\xi-\xi_1)^2\rangle^{2b} }\right\|_{L^{\infty}_{\xi,\tau}},\label{ww2}
      \end{eqnarray}
where $A(\xi,\tau)=\{(\xi_1,\tau_1)\in\mathbb{R}^2; \max\{|\tau-\xi^2|,|\tau_1-\xi_1^2|,|\tau-\tau_1+(\xi-\xi_1)^2|\}=|\tau-\xi^2| \}$.

For $W_3$ we integrate over $\xi$ and $\tau$ first, and then we use the Cauchy-Schwarz and H\"older inequalities to obtain

\begin{equation}\label{ww3}
\begin{split}
W_3^2&\leq \|\phi\|_{L^2}^2\|g\|_{L^2}^2\|h\|_{L^2}^2\\
& \quad \times \left\| \frac{|\xi|^{2s}}{\langle\tau_1-\xi_1^2\rangle^{2b}}\int \int\frac{\chi_{\{B(\xi_1,\tau_1)\}}d\tau d\xi}{\langle \xi_1\rangle^{2s}\langle\tau-\xi^2\rangle^{2b}\langle\xi-\xi_1\rangle^{2s}\langle\tau-\tau_1-(\xi-\xi_1)^2\rangle^{2b} }\right\|_{L^{\infty}_{\xi_1,\tau_1}},
\end{split}
\end{equation}
where $B(\xi_1,\tau_1)=\{(\xi,\tau)\in\mathbb{R}^2;\  \max\{|\tau-\xi^2|,|\tau_1-\xi_1^2|,|\tau-\tau_1+(\xi-\xi_1)^2|\}=|\tau_1-\xi_1^2| \}$.

By symmetry, the estimate for $W_4$ is similar to the $W_3$ estimate.

From the estimates \eqref{ww2} and \eqref{ww3} it suffices to show that
\begin{equation}\label{eq2}
\left\| \frac{|\xi|^{2s}}{\langle\tau-\xi^2\rangle^{2b}}\int \int\frac{\chi_{\{A(\xi,\tau)\}}d\tau_1 d\xi_1}{\langle \xi_1\rangle^{2s}\langle\tau_1-\xi_1^2\rangle^{2b}\langle\xi-\xi_1\rangle^{2s}\langle\tau-\tau_1-(\xi-\xi_1)^2\rangle^{2b} }\right\|_{L^{\infty}_{\xi,\tau}}\leq c,
\end{equation}
\begin{equation}\label{eq3}
\left\| \frac{1}{\langle\tau_1-\xi_1^2\rangle^{2b}}\int \int\frac{|\xi|^{2s}\chi_{\{B(\xi_1,\tau_2)\}}d\tau d\xi}{\langle \xi_1\rangle^{2s}\langle\tau-\xi^2\rangle^{2b}\langle\xi-\xi_1\rangle^{2s}\langle\tau-\tau_1-(\xi-\xi_1)^2\rangle^{2b} }\right\|_{L^{\infty}_{\xi_1,\tau_1}}\leq c.
\end{equation}
To estimate \eqref{eq2} and $\eqref{eq3}$, we will use the algebraic relation
\begin{equation}\label{algebra}
\tau-\xi^2-(\tau-\tau_1-(\xi-\xi_1)^2)-(\tau_1-\xi_1^2)=-2\xi_1(\xi-\xi_1).
\end{equation}
Thus, in $A(\xi,\tau)$ we have 
\begin{equation}\label{3011}
|\xi_1(\xi-\xi_1)|\leq \frac{3}{2}|\tau-\xi^2|\; \; \text{and}
\end{equation}
\begin{equation}
|\tau-\xi^2+2\xi(\xi-\xi_1)|\leq |\tau_1-\xi_1^2|+|(\tau-\tau_1-(\xi-\xi_1)^2)|\leq 2|\tau-\xi^2|.
\end{equation}
We can assume that $|\xi-\xi_1|\geq 1$, since otherwise, the bound reduces to the case $\rho=0$. Thus, we have $\langle \xi_1\rangle^{\rho} \langle \xi-\xi_1\rangle^{\rho}\leq c |\xi_1|^{\rho}|\xi-\xi_1|^{\rho}.$

Using Lemma \ref{lemagtv} and \eqref{3011}, we see that
\begin{equation}\label{14091}
\begin{split}
&\frac{1}{\langle\tau-\xi^2\rangle^{2b}\langle\xi\rangle^{2\rho}}\int\int_{A}\frac{(|\xi_1||\xi-\xi_1|)^{2\rho}d\tau_1d\xi_1}{\langle\tau_1-\xi_1^2\rangle^{2b}\langle\tau-\tau_1-(\xi-\xi_1)^2\rangle^{2b}}\\
&\quad \leq  \frac{c}{\langle\tau-\xi^2\rangle^{2b}\langle\xi\rangle^{2\rho}}\int\frac{\langle\tau-\xi^2\rangle^{2\rho}d\xi_1}{\langle\tau-\xi^2+2\xi_1(\xi-\xi_1)\rangle^{4b-1}},
\end{split}
\end{equation}
where we have used that $\frac{1}{4}<b<\frac{1}{2}$.

Set $\eta=\tau-\xi^2+2\xi_1(\xi-\xi_1)$. Then $\xi_1=\frac{1}{2}(\xi\pm (2\tau-\xi^2-2\eta)^{\frac{1}{2}})$, $|2\xi_1-\xi|=|2\tau-\xi^2-2\eta|^{\frac{1}{2}}$ and $d\xi_1=\frac{c}{|\tau-\eta-\frac{\xi^2}{2}|^{\frac{1}{2}}}d\eta$. Substituting into right hand side \eqref{14091} and using Lemma \ref{Hol}, we estimate the right hand side of \eqref{14091} by 
 \begin{eqnarray*}
	& & \frac{c}{\langle \xi\rangle^{2\rho}\langle\tau-\xi^2\rangle^{2b-2\rho}}\int_{|\eta|\leq 2|\tau-\xi^2|}\frac{d\eta}{\langle\eta\rangle^{4b-1}|\tau-\eta-\frac{\xi^2}{2}|^{\frac{1}{2}}}\\
	& &\quad \leq \frac{c\langle\tau-\xi^2\rangle^{2-4b}}{\langle\tau-\xi^2\rangle^{2b-2\rho}\langle\tau-\xi^2\rangle^{\frac{1}{2}}}\leq c\langle\tau-\xi^2\rangle^{2-6b+2\rho-\frac{1}{2}}.
\end{eqnarray*}
This expression is bounded, provided $b\geq \frac{1}{4}+\frac{\rho}{3}$. Note that as $\rho<\frac{3}{4}$, we can choice $b(\rho)$ depending of $\rho$ such that $b(\rho)<\frac{1}{2}$.

Now we prove \eqref{eq3} in the case $\rho  \in(\frac{1}{2},\frac{3}{4})$. In $B(\xi_1,\tau_1)$, using \eqref{algebra}, we have that
\begin{equation}
|\xi_1(\xi-\xi_1)|\leq\frac{3}{2}|\tau_1-\xi_1^2|\; \; \text{and}
\end{equation}
\begin{equation}
|\tau_1-\xi_1^2+2\xi_1(\xi-\xi_1)|\leq |\tau-\xi^2|+|\tau-\tau_1-(\xi-\xi_1)^2|\leq 2|\tau_1-\xi_1^2|.
\end{equation}
From Lemma \ref{lemagtv}, we obtain
\begin{equation*}
\begin{split}
& \frac{1}{\langle\tau_1-\xi_1^2\rangle^{2b}}\int\int_{B}\frac{(|\xi_1||\xi-\xi_1|)^{2\rho}d\xi d\tau}{\langle\xi\rangle^{2\rho}\langle\tau-\xi^2\rangle^{2b}\langle\tau-\tau_1-(\xi-\xi_1)^2\rangle^{2b}}\\
& \quad \leq  \frac{c}{\langle\tau-\xi_1^2\rangle^{2b-2\rho}}\int_{D} \frac{d\xi}{\langle\xi\rangle^{2\rho}\langle\tau_1-\xi_1^2+2\xi_1(\xi_1-\xi)\rangle^{4b-1}},
\end{split}
\end{equation*}
 where $D=D(\xi_1,\tau_1)=\{\xi\in\mathbb{R};\,|\tau_1-\xi_1^2+2\xi_1(\xi_1-\xi)|\leq 2|\tau_1-\xi_1^2|\}$ and $\frac{1}{4}<b<\frac{1}{2}$. 
Now we divide $D$ into two parts $D_1$ and $D_2$, where, $D_1=\big\{\xi\in D;\,|2\xi_1(\xi_1-\xi)|\leq \frac{|\tau_1-\xi_1^2|}{2}\big\},\ 
D_2=\big\{\xi\in D;\,\frac{|\tau_1-\xi_1^2|}{4}\leq|\xi_1(\xi_1-\xi)|\leq\frac{3|\tau_1-\xi_1^2|}{2}\big\}.$

In $D_1$ we have $|\tau_1-\xi_1^2|\leq 2 |\tau_1-\xi_1^2+2\xi_1(\xi_1-\xi)|$, it follows that
\begin{equation}\label{14094}
\frac{1}{\langle\tau_1-\xi_1^2\rangle^{2b-2\rho}}\int_{D_1} \frac{d\xi}{\langle\xi\rangle^{2\rho}\langle\tau_1-\xi_1^2+2\xi_1(\xi_1-\xi)\rangle^{4b-1}}\leq \frac{1}{\langle\tau_1-\xi_1^2\rangle^{2b-2\rho+4b-1}}\int \frac{d\xi}{\langle\xi\rangle^{2\rho}}.
\end{equation}

This expression is bounded, provided $\rho>\frac{1}{2}$ and $b\geq\frac{1}{6}+\frac{\rho}{3}$. Note that as $\rho<\frac{3}{4}$, we can choice $b(\rho)<\frac{1}{2}$.

We subdivide $D_2$, into three pieces, i.e., $D_{21}\cup D_{22}\cup D_{23}$, where

\begin{equation*}
 D_{21}=\big\{\xi\in D_2;\, \frac{|\xi|}{4}\leq |\xi_1|\leq 100|\xi|\big\},\ D_{22}=\big\{\xi\in D_2;\, 1\leq |\xi_1|\leq \frac{|\xi|}{4}\big\}\ \text{and}
 \end{equation*}
 \begin{equation*}
 D_{23}=\{\xi\in D_2;\, 100|\xi|\leq |\xi_1|\}.
\end{equation*}

In $D_{21}$ we have $|\xi|^2\sim |\xi_1|^2\geq c \langle\tau_1-\xi_1^2\rangle$. Set $\eta=\tau_1-\xi_1^2+2\xi_1(\xi_1-\xi)$. Then $d\eta=-2\xi_1d\xi$. Using Lemma \ref{Hol}, we obtain
\begin{eqnarray*}
	& &\frac{1}{\langle\tau_1-\xi_1^2\rangle^{2b-2\rho}}\int_{D_{21}} \frac{d\xi}{\langle\xi\rangle^{2\rho}\langle\tau_1-\xi_1^2+2\xi_1(\xi_1-\xi)\rangle^{4b-1}}\\
	& &\quad \leq \frac{c}{\langle\tau_1-\xi_1^2\rangle^{2b-\rho}}\int_{|\eta|\leq 2|\tau_1-\xi_1^2|} \frac{d\eta}{\langle\eta\rangle^{4b-1}|\xi_1|} \leq c \frac{1}{\langle\tau_1-\xi_1^2\rangle^{2b-\rho}}\frac{1}{\langle\tau_1-\xi_1^2\rangle^{4b-2}},
\end{eqnarray*}
which is bounded provided $b\geq \frac{1}{3}+\frac{\rho}{6}$.

In $D_{22}$, we have $\langle\tau_1-\xi_1^2\rangle\geq \frac{2}{3}|\xi_1||\xi-\xi_1|\geq \frac{2}{3} |\xi_1|(|\xi|-|\xi_1|)\geq\frac{2}{3}|\xi_1|(4|\xi_1|-\xi_1)=2|\xi_1|^2.$ 

Then, by Lemma \ref{Hol}, we obtain
\begin{equation*}
\begin{split}
	 &\frac{1}{\langle\tau_1-\xi_1^2\rangle^{2b}}\!\int_{D_{22}}\! \frac{(|\xi_1||\xi-\xi_1|)^{2\rho}d\xi}{\langle\xi\rangle^{2\rho}\langle\tau_1-\xi_1^2+2\xi_1(\xi_1-\xi)\rangle^{4b-1}} \leq  \frac{c}{\langle\tau_1-\xi_1^2\rangle^{2b}}\!\int_{|\eta|\leq |\tau_1-\xi_1^2|}\!\frac{|\xi_1|^{2\rho}d\eta}{|\xi_1|\langle\eta\rangle^{4b-1}}\\
	 &\quad \leq\frac{c}{\langle\tau_1-\xi_1^2\rangle^{2b-\rho+\frac{1}{2}}}\int_{|\eta|\leq |\tau_1-\xi_1^2|}\frac{d\eta}{\langle\eta\rangle^{4b-1}}\leq \frac{c}{\langle\tau_1-\xi_1^2\rangle^{6b-\frac{3}{2}-\rho}},
	 \end{split}
\end{equation*}
 which is bounded provided $b\geq\frac{1}{4}+\frac{\rho}{6}$.
In $D_{23}$ we use $|\xi_1|^2 \sim|\tau_1-\xi_1^2|$ to obtain

\begin{equation*}
\begin{split}
	 &\frac{1}{\langle\tau_1-\xi_1^2\rangle^{2b-2\rho}}\int_{D_{23}} \frac{d\xi}{\langle\xi\rangle^{2\rho}\langle\tau_1-\xi_1^2+2\xi_1(\xi_1-\xi)\rangle^{4b-1}}\leq \frac{c}{\langle\tau_1-\xi_1^2\rangle^{2b-2\rho}}\int_{|\eta|\leq 2|\tau_1-\xi_1^2|} \frac{d\eta}{|\xi_1|\langle\eta\rangle^{4b-1}}\\
	 &\quad \leq \frac{c}{\langle\tau_1-\xi_1^2\rangle^{2b-2\rho}}\frac{1}{\langle\tau_1-\xi_1^2\rangle^{\frac{1}{2}}}\frac{1}{\langle\tau_1-\xi_1^2\rangle^{4b-2}}\leq \frac{c}{\langle\tau_1-\xi_1^2\rangle^{6b-2\rho-\frac{3}{2}}},
	 \end{split}
\end{equation*} 
which is bounded provided $b\geq \frac{\rho}{3}+\frac{1}{4}$. Since $\rho<\frac{3}{4}$, we can choice $b=b(\rho)$ such that $b<\frac{1}{2}$. 

The inequality \eqref{eq3} in the case $\rho\in (0,\frac{1}{2})$ follows interpolate the cases $\rho=0$ and $\rho\in(\frac{1}{2},\frac{3}{4})$. This finish the proof of the part (a) of the lemma.

Arguing as in the part (a) of proof, to obtain Lemma \ref{estimativan1} (b), it suffices to prove 

\begin{equation}\label{15091}
\left\|\frac{1}{\langle\tau-\xi^2\rangle^{2b}\langle\tau\rangle^{\rho}}\int\int_{A}\frac{(|\xi_1||\xi-\xi_1|)^{2\rho}d\xi_1d\tau_1}{\langle\tau_1-\xi_1^2\rangle^{2b}\langle\tau-\tau_1-(\xi-\xi_1)^2\rangle^{2b}} \right\|_{L^{\infty}_{\xi,\tau}}\leq c,
\end{equation}
where $A(\xi,\tau)=\{(\xi_1,\tau_1)\in\mathbb{R}^2;\,\max\{|\tau-\xi^2|,|\tau_1-\xi_1^2|,|\tau-\tau_1-(\xi-\xi_1)^2|\}=|\tau-\xi^2|\},$
and
\begin{equation}\label{15092}
\left\|\frac{1}{\langle\tau_1-\xi_1^2\rangle^{2b}}\int\int_{B}\frac{(|\xi_1||\xi-\xi_1|)^{2\rho}d\xi d\tau}{\langle\tau\rangle^{\rho}\langle\tau-\xi^2\rangle^{2b}\langle\tau-\tau_1-(\xi-\xi_1)^2\rangle^{2b}}\right\|_{L_{\xi_1,\tau_1}^{\infty}}\leq c,
\end{equation}
where $B=B(\xi_1,\tau_1)=\{(\xi,\tau)\in\mathbb{R}^2;\, \max\{|\tau-\xi^2|,|\tau_1-\xi_1^2|,|\tau-\tau_1-(\xi-\xi_1)^2|\}=|\tau_1-\xi_1^2|\}.$

The inequality \eqref{15091} can be obtained in the same way of the estimate \eqref{eq2}. To estimate \eqref{15092}, we can assume $10|\tau|<|\xi|^2$, since otherwise, the bound reduces to the inequality \eqref{eq3}. In this case we have $\langle\tau-\xi^2 \rangle\sim \langle\xi^2\rangle$. Thus, we need to show
\begin{equation}\label{15093}
\frac{1}{\langle\tau_1-\xi_1^2\rangle^b}\int\int_{B}\frac{(|\xi_1||\xi-\xi_1|)^{2\rho}d\xi d\tau}{\langle\tau\rangle^{\rho}\langle\xi^2\rangle^{2b}\langle\tau-\tau_1-(\xi-\xi_1)^2\rangle^{2b}}\leq c.
\end{equation}
Using the definition of $B$ and \eqref{algebra}, we see that $|\xi_1||\xi-\xi_1|\leq \langle \tau_1-\xi_1^2\rangle$, it follows that
\begin{equation}\label{15095}
\frac{(|\xi_1||\xi-\xi_1|)^{\frac{4}{3}\rho}}{\langle\tau_1-\xi_1^2\rangle^{2b}}\leq \frac{1}{\langle\tau_1-\xi_1^2\rangle^{2b-\frac{4\rho}{3}}}\leq c,
\end{equation}  
where we have used that $\rho<\frac{3}{2}b$. Using \eqref{15095} and Lemma \ref{lemagtv}, we see that the left hand side of \eqref{15093} is bounded by
\begin{equation*}
	 c \int\int_{B}\frac{(|\xi_1||\xi-\xi_1|)^{\frac{2}{3}\rho}d\xi d\tau}{\langle\tau\rangle^{\rho}\langle\xi^2\rangle^{2b}\langle\tau-\tau_1-(\xi-\xi_1)^2\rangle^{2b}}\leq c \int_{\{\xi\in\R;\ \exists\ \tau; (\xi,\tau)\in B\}}\frac{(|\xi_1||\xi-\xi_1|)^{\frac{2}{3}\rho}d\xi}{\langle\tau_1+(\xi-\xi_1)^2\rangle^{\rho+2b-1}\langle\xi^2\rangle^{2b}}.
\end{equation*}
To estimate the last integral, we divide the integration domain in two parts, i.e., $D_1=\{\xi\in \mathbb{R};|\xi_1|>2|\xi-\xi_1|\ \text{or}\ |\xi-\xi_1|>2|\xi_1|\}$ and $D_2=\{\xi\in \mathbb{R};\frac{1}{2}|\xi-\xi_1|\leq |\xi_1|\leq 2|\xi-\xi_1|\}.$

In $D_1$ we have  $\frac{(|\xi_1||\xi-\xi_1|)^{\frac{2}{3}\rho}}{\langle\xi\rangle^{4b}}\leq \frac{c}{\langle\xi\rangle^{4b-\frac{4}{3}\rho}} \leq c
$, for $\rho<3b$. 
Thus, Lemma \ref{lemanovo} implies 

\begin{equation*}
\int_{D_1}\frac{(|\xi_1||\xi-\xi_1|)^{\frac{2}{3}\rho}d\xi}{\langle\tau_1+(\xi-\xi_1)^2\rangle^{\rho+2b-1}\langle\xi^2\rangle^{2b}}\leq  c\int\frac{d\xi}{\langle\tau_1+(\xi-\xi_1)^2\rangle^{\rho+2b-1}} \leq c,
\end{equation*}
where we have used that $b>\frac{3}{4}-\frac{\rho}{2}$.

Now we subdivide $D_2$ in two pieces, i.e., $D_{21}\cup D_{22}$, where $D_{21}=\{\xi\in D_2;\ 10|\xi_1||\xi-\xi_1|\leq |\tau_1-\xi_1^2|\}\ \text{and}\ 
\ D_{22}=\{\xi\in D_2;\  |\xi_1||\xi-\xi_1|\sim |\tau_1-\xi_1^2|\}$.

In $D_{21}$ we have $|\tau_1-\xi_1^2-2\xi_1(\xi-\xi_1)|\sim |\tau_1-\xi_1^2|$. Then, by Lemma \ref{lemagtv},
 \begin{equation*}
\begin{split}
	 &\frac{c}{\langle\tau_1-\xi_1^2\rangle^{2b}}\int\int_{B\cap D_{21}\times \mathbb{R} }\frac{(|\xi_1||\xi-\xi_1|)^{2\rho}d\xi d\tau}{\langle\tau\rangle^{\rho}\langle\xi^2\rangle^{2b}\langle\tau-\tau_1-(\xi-\xi_1)^2\rangle^{2b}}\\
	 &\quad\leq \frac{c}{\langle\tau_1-\xi_1^2\rangle^{2b-2\rho}}\int\int_{B\cap D_{21}\times \mathbb{R}}\frac{d\xi d\tau}{\langle \tau+\xi^2\rangle^{2b}\langle\tau-\tau_1-(\xi-\xi_1)^2\rangle^{2b}}\\
	 & \quad\leq \frac{c}{\langle\tau_1-\xi_1^2\rangle^{2b-2\rho}}\int_{D_{21}}\frac{d\xi}{\langle \tau_1-\xi_1^2-2\xi_1(\xi-\xi_1)\rangle^{4b-1}}\leq  \frac{c}{\langle\tau_1-\xi_1^2\rangle^{2b-2\rho}}\int_{|\xi|\leq 3|\xi_1|}\frac{d\xi}{\langle \tau_1-\xi_1^2\rangle^{4b-1}}\\
	&\quad \leq c  \int_{|\xi|\leq 3|\xi_1|}\frac{d\xi}{\langle \tau_1-\xi_1^2\rangle^{6b-1-2\rho}}\leq\frac{c|\xi_1|}{\langle\tau_1-\xi_1^2\rangle^{6b-1-2\rho}}\leq \frac{c}{\langle\tau_1-\xi_1^2\rangle^{6b-1-2\rho-\frac{1}{2}}},
	\end{split}
\end{equation*}
which is bounded provided $\rho\leq 3b-\frac{3}{4}$. Since $\rho<\frac{3}{4}$ we can choice $b$ depending of $\rho$, such that $b<\frac{1}{2}$. In $D_{22}$ we have $|\xi_1|^2\sim |\tau_1-\xi_1^2|$, and consequently
\begin{equation*}
\begin{split}
	 &\frac{1}{\langle\tau_1-\xi_1^2\rangle^{2b}}\!\!\int_{D_{22}}\!\! \frac{(|\xi_1||\xi-\xi_1|)^{2\rho}d\xi}{\langle\tau_1-\xi_1^2+2\xi_1(\xi_1-\xi)\rangle^{4b-1}}\leq \frac{c}{\langle\tau_1-\xi_1^2\rangle^{2b}}\int_{D_{22}} \frac{|\xi_1|^{4\rho}d\xi}{\langle\tau_1-\xi_1^2+2\xi_1(\xi_1-\xi)\rangle^{4b-1}}\\
	 &\quad\leq \frac{c}{\langle\tau_1-\xi_1^2\rangle^{2b-2\rho}}\int_{|\eta|\leq 2|\tau_1-\xi_1^2|} \frac{d\eta}{|\xi_1|\langle\eta\rangle^{4b-1}}\leq \frac{c}{\langle\tau_1-\xi_1^2\rangle^{6b-2\rho-\frac{3}{2}}},
	 \end{split}
\end{equation*} 
which is bounded provided $\rho\leq 3b-\frac{3}{4}$, and hence the proof of Lemma \ref{estimativan1} is completed.
\subsection{Proof of Lemma \ref{estimativan2}}
Let $\rho=-s\in [0,\frac{3}{4})$. For $u,v\in X^{s,b}=X^{-\rho,b}$ define
$$g(\xi,\tau)=\langle\tau-\xi^2\rangle^b\langle\xi\rangle^{-\rho}\hat{u}(\xi,\tau),\ h(\xi,\tau)=\langle\tau-\xi^2\rangle^b\langle\xi\rangle^{-\rho}\hat{v}(\xi,\tau).$$Then $\|g\|_{L_{\xi}^2L_{\tau}^2}=\|u\|_{X^{s,b}},\ \|h\|_{L_{\xi}^2L_{\tau}^2}=\|v\|_{X^{s,b}}$. Using  $\hat{\overline{v}}(\xi,\tau)=\frac{\langle\xi\rangle^{\rho}}{\langle\tau+\xi^2\rangle^b}\overline{h}(-\xi,-\tau)$, we have
\begin{equation*}
	\|N_2(u,\overline{v})\|_{X^{s,-b}}=\sup_{\|\phi\|_{L^2}\leq 1}W(g,h,\phi),
\end{equation*}
where
\begin{equation*}
W(g,h,\phi):=\int_{\mathbb{R}^4}\frac{\langle \xi\rangle^s}{\langle\tau-\xi^2\rangle^{b}}\frac{\overline{h}(-\xi_1,-\tau_1)}{\langle \xi_1\rangle^{s}\langle\tau_1+\xi_1^2\rangle^b}\frac{g(\xi-\xi_1,\tau-\tau_1)}{\langle \xi-\xi_1\rangle^{s}\langle\tau-\tau_1-(\xi-\xi_1)^2\rangle^b}\phi(\xi,\tau)d\xi_1d\tau_1d\xi d\tau. 
\end{equation*}
Initially, we treat the case $\rho=0$. Integrating first in $\xi,\tau$, changing of variables $\tau_2=\tau-\tau_1$, $\xi_2=\xi-\xi_1$, and using Cauchy-Schwarz and H\"older inequalities, we obtain
\begin{equation*}
\begin{split}
W^2(g,h,\phi) &\leq c \|\phi\|_{L^2}^2\|g\|_{L^2}^2\|h\|_{L^2}^2\\
&\quad\times \left\|\frac{1}{\langle \tau_2-\xi_2^2\rangle^{2b}}\int \int \frac{d\xi d\tau}{\langle \tau-\xi^2\rangle^{2b}\langle \tau-\tau_2+(\xi-\xi_2)^2\rangle^{2b}}\right\|_{L_{\xi_2,\tau_2}^{\infty}}.
\end{split}
\end{equation*}
Thus we need to show 
\begin{equation}\label{mc1}
\sup_{(\xi_2,\tau_2)\in \mathbb{R}^2}\frac{1}{\langle \tau_2-\xi_2^2\rangle^{2b}}\int \int \frac{d\xi d\tau}{\langle \tau-\xi^2\rangle^{2b}\langle \tau-\tau_2+(\xi-\xi_2)^2\rangle^{2b}}\leq c.
\end{equation}
Applying Lemmas \ref{lemagtv} and \ref{lemanovo}, we see that the left hand side of \eqref{mc1} is bounded by
\begin{equation*}
c \int \frac{d\xi}{\langle 2\xi^2-2\xi\xi_2-\tau_2\xi_2^2\rangle^{4b-1}}\leq c,
\end{equation*}
where we have used that $\frac{3}{8}<b<\frac{1}{2}$.

Now we treat the case $\rho\in(0,\frac{1}{4})$. In this case we shall assume that $|\xi_1|\geq 10$ and $|\xi-\xi_1|\geq 10$, since otherwise, the bound reduces to the case $\rho=0$. We write 
\begin{eqnarray*}
	W(g,h,\phi)&=&\int_{\mathbb{R}^4}\frac{\langle \xi\rangle^s}{\langle\tau-\xi^2\rangle^{b}}\frac{\overline{h}(-\xi_1,-\tau_1)}{\langle \xi_1\rangle^{s}\langle\tau_1+\xi_1^2\rangle^b}\frac{g(\xi-\xi_1,\tau-\tau_1)}{\langle \xi-\xi_1\rangle^{s}\langle\tau-\tau_1-(\xi-\xi_1)^2\rangle^b}\phi(\xi,\tau)d\xi_1d\tau_1d\xi d\tau  \nonumber \\
	&=&\int_{\mathcal{R}_1}+\int_{\mathcal{R}_2}+\int_{\mathcal{R}_3}+\int_{\mathcal{R}_4}+\int_{\mathcal{R}_5}:= W_1+W_2+W_3+W_4+W_5,\nonumber
\end{eqnarray*}
where 
\begin{equation*}
\mathcal{R}_1=\{(\xi,\xi_1,\tau,\tau_1)\in\mathbb{R}^4;\, |\xi_1|\geq 10, |\xi|\leq 1 \},
\end{equation*}
\begin{equation*}
\begin{split}
&\mathcal{R}_2=\big\{(\xi,\xi_1,\tau,\tau_1)\in\mathbb{R}^4;\, |\xi_1|\geq 10, |\xi|\geq 1,\\
&\ \ \ \ \ \ \ \ \ \ \ \ \ \ \ \ \ \ \ \ \ \ \ \ \ \ \ \  \ \ \ \ \ \ \ \ \ \ \  \max\{|\tau-\xi^2|,|\tau_1+\xi_1^2|, |\tau-\tau_1-(\xi-\xi_1)^2|\}=|\tau_1+\xi_1^2| \big\},
\end{split}
\end{equation*}

\begin{equation*}
\begin{split}
&\mathcal{R}_3=\big\{(\xi,\xi_1,\tau,\tau_1)\in\mathbb{R}^4;\, |\xi_1|\geq 10, 1<|\xi|\!<|\xi_1|,\\
&\ \ \ \ \ \ \ \ \ \ \ \ \ \ \ \ \ \ \ \ \ \ \ \ \ \ \ \  \ \ \ \ \ \ \ \ \ \ \ \max\{|\tau-\xi^2|,|\tau_1+\xi_1^2|, |\tau-\tau_1-(\xi-\xi_1)^2|\}=|\tau-\xi^2|\big \},
\end{split}
\end{equation*}
\begin{equation*}
\begin{split}
&\mathcal{R}_4=\big\{(\xi,\xi_1,\tau,\tau_1)\in\mathbb{R}^4;\, |\xi_1|\geq 10, |\xi|\geq|\xi_1|,\\
&\ \ \ \ \ \ \ \ \ \ \ \ \ \ \ \ \ \ \ \ \ \ \ \ \ \ \ \  \ \ \ \ \ \ \ \ \ \ \  \max\{|\tau-\xi^2|,|\tau_1+\xi_1^2|, |\tau-\tau_1-(\xi-\xi_1)^2|\}=|\tau-\xi^2| \big\},\\
&\mathcal{R}_5=\big\{(\xi,\xi_1,\tau,\tau_1)\in\mathbb{R}^4;\, |\xi|\geq 1, |\xi_1|\geq 10,\\&\ \ \ \ \ \ \ \ \ \ \ \ \ \ \ \ \ \ \ \ \ \ \ \ \ \ \ \  \ \ \ \ \ \ \max\{|\tau-\xi^2|,|\tau_1+\xi_1^2|, |\tau-\tau_1-(\xi-\xi_1)^2|\}=|\tau-\tau_1-(\xi-\xi_1)^2| \big\}.
\end{split}
\end{equation*}
Thus, we need to show the following estimates

\begin{equation}\label{18091}
\sup_{|\xi_1|>10,\ \tau_1\in \mathbb{R}}\frac{1}{\langle \tau_1+\xi_1^2\rangle^{2b}}\int\int_{|\xi|<1}\frac{(|\xi_1||\xi-\xi_1|)^{2\rho}d\xi d\tau}{\langle\xi\rangle^{2\rho}\langle\tau-\tau_1-(\xi-\xi_1)^2\rangle^{2b}\langle\tau-\xi^2\rangle^{2b}}\leq c,
\end{equation}
\begin{equation}\label{faltava}
\sup_{|\xi_1|>10,\ \tau_1\in \mathbb{R}}\frac{1}{\langle \tau_1+\xi_1^2\rangle^{2b}}\int\int\frac{\chi_{A(\xi_1,\tau_1)}(|\xi_1||\xi-\xi_1|)^{2\rho}d\xi d\tau}{\langle\xi\rangle^{2\rho}\langle\tau-\tau_1-(\xi-\xi_1)^2\rangle^{2b}\langle\tau-\xi^2\rangle^{2b}}\leq c,
\end{equation}
\begin{equation}\label{18092}
\sup_{|\xi|>1, \tau\in \mathbb{R}}\frac{1}{\langle\tau-\xi^2\rangle^{2b}\langle\xi\rangle^{2\rho}}\int\int\frac{\chi_{B(\xi,\tau)}(|\xi_1||\xi-\xi_1|)^{2\rho}d\xi_1d\tau_1}{\langle\tau_1+\xi_1^2\rangle^{2b}\langle\tau-\tau_1-(\xi-\xi_1)^2\rangle^{2b}}\leq c,
\end{equation}
\begin{equation}\label{18093}
\sup_{|\xi|>1, \tau\in \mathbb{R}}\frac{1}{\langle\tau-\xi^2\rangle^{2b}\langle\xi\rangle^{2\rho}}\int\int\frac{\chi_{C(\xi,\tau)}(|\xi_1||\xi-\xi_1|)^{2\rho}d\xi_1d\tau_1}{\langle\tau_1+\xi_1^2\rangle^{2b}\langle\tau-\tau_1-(\xi-\xi_1)^2\rangle^{2b}}\leq c,
\end{equation}
\begin{equation}\label{18094}
\sup_{|\xi_1|>1,\tau_1\in\mathbb{R}}\frac{1}{\langle\tau_1-\xi_1^2\rangle^{2b}}\int\int\frac{\chi_{D(\xi_1,\tau_1)}(|\xi_1||\xi-\xi_1|)^{2\rho}d\tau d\xi}{\langle \xi\rangle^{2\rho}\langle\tau-\xi^2\rangle^{2b}\langle\tau-\tau_1+(\xi-\xi_1)^2\rangle^{2b}}\leq c,
\end{equation}
where
\begin{equation*}
\begin{split}
&A(\xi_1,\tau_1)=\{(\xi,\tau)\in\mathbb{R}^2;\,|\xi|\geq 1,\, \max\{|\tau-\xi^2|,|\tau_1+\xi_1^2|,|\tau-\tau_1-(\xi-\xi_1)^2|\}=|\tau_1+\xi_1^2|\},\\
& B(\xi,\tau)=\big\{(\xi_1,\tau_1)\in\mathbb{R}^2;\,|\xi_1|>10, |\xi|\leq |\xi_1|,\\
&\ \ \ \ \ \ \ \ \ \ \ \ \ \ \ \ \ \ \ \ \ \ \ \ \ \ \ \ \ \ \ \ \ \  \ \ \ \ \ 
\max\{|\tau-\xi^2|,|\tau_1+\xi_1^2|,|\tau-\tau_1-(\xi-\xi_1)^2|\}=|\tau-\xi^2|\big\},\\
&C(\xi,\tau)=\big\{(\xi_1,\tau_1)\in\mathbb{R}^2;\, |\xi_1|>10,\ |\xi| \geq |\xi_1|,\\
&\ \ \ \ \ \ \ \ \ \ \ \ \ \ \ \ \ \ \ \ \ \ \ \ \ \ \ \ \ \ \ \ \ \  \ \ \ \ \  \max\{|\tau-\xi^2|,|\tau_1+\xi_1^2|,|\tau-\tau_1-(\xi-\xi_1)^2|\}=|\tau-\xi^2|\big\},
\end{split}
\end{equation*}
\begin{equation*}
\begin{split}
	&D(\tau_1,\xi_1)=\big\{(\xi,\tau)\in \mathbb{R}^2;\, |\xi|\!\geq\! 1,\ |\xi-\xi_1|\!\geq\! 10,\\
&\ \ \ \ \ \ \ \ \ \ \ \ \ \ \ \ \ \ \ \ \ \ \ \ \ \ \ \ \ \ \ \ \ \  \ \ \ \ \ 	\max\{|\tau-\xi^2|,|\tau_1-\xi_1^2|,|\tau-\tau_1+(\xi-\xi_1)^2|\}\!=\!|\tau_1-\xi_1^2|\big\}.
	\end{split}
	\end{equation*}
First, we obtain \eqref{18091}.  Using Lemma \ref{lemagtv}, we see that the left hand side of \ref{lemagtv} is bounded by

\begin{eqnarray}
		& &\sup_{|\xi_1|>10,\ \tau_1\in \mathbb{R}} \frac{c\chi_{\{|\xi_1|>10\}}|\xi_1|^{4\rho}}{\langle \tau_1+\xi_1^2\rangle^{2b}}\int\int_{|\xi|<1}\frac{d\xi d\tau}{\langle\tau-\tau_1-(\xi-\xi_1)^2\rangle^{2b}\langle\tau-\xi^2\rangle^{2b}}\nonumber\\
	& &\quad \leq \sup_{|\xi_1|>10,\ \tau_1\in \mathbb{R}} \frac{c\chi_{\{|\xi_1|>10\}}|\xi_1|^{4\rho}}{\langle \tau_1+\xi_1^2\rangle^{2b}}\int_{|\xi|<1}\frac{d\xi}{\langle\tau_1+\xi_1^2-2\xi\xi_1\rangle^{4b-1}}.\label{0212}
\end{eqnarray}
Changing variables $\eta=\tau_1+\xi_1^2-2\xi\xi_1$, then $d\eta=-2\xi_1d\xi$ and $|\eta|\leq |\tau_1+\xi_1^2|+2|\xi_1|$. Substituting into \eqref{0212}, we see that \eqref{0212} is controlled by
\begin{equation*}
 \sup_{|\xi_1|>10,\ \tau_1\in \mathbb{R}}\frac{c\chi_{\{|\xi_1|>10\}}|\xi_1|^{4\rho}}{\langle \tau_1+\xi_1^2\rangle^{2b}|\xi_1|}\int_{|\eta|\leq |\tau_1+\xi_1^2|+2|\xi_1|}\frac{d\eta}{\langle\eta\rangle^{4b-1}}\\
\leq c\frac{\chi_{\{|\xi_1|>10\}}|\xi_1|^{4\rho}}{\langle \tau_1+\xi_1^2\rangle^{2b}|\xi_1|}(|\tau_1+\xi_1^2|^{2-4b}+|\xi_1|^{2-4b}),
\end{equation*}
which is bounded provided $b>\frac{1}{3}$ and $\rho\leq b-\frac{1}{4}$.

To obtain \eqref{faltava}, we use the following algebraic relation
\begin{equation}\label{algb}
\tau-\tau_1-(\xi-\xi_1)^2+(\tau_1+\xi_1^2)-(\tau-\xi^2)=2\xi\xi_1.
\end{equation}
By Lemma \ref{lemagtv}, we have that the left hand side of \eqref{faltava} is bounded by
\begin{equation}\label{novamente}
\sup_{|\xi_1|>10,\ \tau_1\in \mathbb{R}}\frac{c|\xi_1|^{4\rho}}{\langle\tau_1+\xi_1^2\rangle^{2b}}\int\frac{d\xi}{\langle\tau_1+\xi_1^2-2\xi\xi_1\rangle^{4b-1}}.
\end{equation}
Set $\eta=\tau+\xi_1^2-2\xi\xi_1$. Then $d\eta=-2\xi_1d\xi$. By \eqref{algb}, we have $|\eta|\leq c |\tau_1+\xi_1^2|$ in $A(\xi_1,\tau_1)$. Substituting into \eqref{novamente}, we obtain that \eqref{novamente} is bounded by $
\frac{c|\xi_1|^{4\rho-1}}{\langle\tau_1+\xi_1^2\rangle^{6b-2}},
$ which is bounded provided $b>\frac{1}{3}$ and $\rho\leq \frac{1}{4}$.
Now we obtain \eqref{18092}. By Lemma \ref{lemagtv},we see that left hand side of \eqref{18092} is bounded by
\begin{eqnarray}
& & \sup_{|\xi|>1, \tau\in \mathbb{R}}\frac{c\chi_{\{|\xi|>1\}}}{\langle\tau-\xi^2\rangle^{2b}\langle\xi\rangle^{2\rho}}\int\frac{\chi_{\{\xi_1;\exists \tau_1; (\xi_1,\tau_1)\in B(\xi,\tau)\}}\langle\xi_1\rangle^{4\rho}d\xi_1}{\langle\tau-\xi^2+2\xi\xi_1\rangle^{4b-1}}\nonumber\\
 & &\quad \leq \sup_{|\xi|>1, \tau\in \mathbb{R}}\frac{c\chi_{\{|\xi|>1\}}}{\langle\tau-\xi^2\rangle^{2b-4\rho}\langle\xi\rangle^{2\rho}}\int\frac{\chi_{\{\xi_1;\exists \tau_1; (\xi_1,\tau_1)\in B(\xi,\tau)\}}d\xi_1}{\langle\tau-\xi^2+2\xi\xi_1\rangle^{4b-1}}.\label{rio1}
\end{eqnarray}
Set $\eta=\tau-\xi^2+2\xi\xi_1$, then $d\eta=2\xi d \xi_1$ and $|\eta|\leq 2|\tau-\xi^2|$. Substituting into \eqref{rio1}, we obtain that \eqref{rio1} is bounded by

\begin{equation}
 \sup_{|\xi|>1, \tau\in \mathbb{R}}\frac{c\chi_{\{|\xi|>1\}}}{\langle\tau-\xi^2\rangle^{2b-4\rho}\langle\xi\rangle^{2\rho}|\xi|}\int_{|\eta|<2|\tau-\xi^2|}\frac{d\eta}{\langle\eta\rangle^{4b-1}}\leq \frac{\sup_{|\xi|>1, \tau\in \mathbb{R}}c}{\langle\tau-\xi^2\rangle^{2b-4\rho}\langle\tau-\xi^2\rangle^{4b-2}}.
\end{equation}

This expression is bounded, provided $\rho\leq\frac{3}{2}b-\frac{1}{2}$.

Now we estimate \eqref{18093}. By Lemma \ref{lemagtv} we see that left hand side of \eqref{18093} is bounded by
\begin{eqnarray}
& &\sup_{|\xi|>1, \tau\in \mathbb{R}}\frac{\chi_{\{|\xi|>1\}}}{\langle\tau-\xi^2\rangle^{2b}\langle\xi\rangle^{2\rho}}\int\int\frac{\chi_{C(\xi,\tau)}(|\xi_1||\xi-\xi_1|)^{\rho}(|\xi_1||\xi-\xi_1|)^{\rho}d\xi_1d\tau_1}{\langle\tau_1+\xi_1^2\rangle^{2b}\langle\tau-\tau_1-(\xi-\xi_1)^2\rangle^{2b}}\nonumber\\
& &\quad \leq\sup_{|\xi|>1, \tau\in \mathbb{R}} \frac{c\chi_{\{|\xi|>1\}}\langle\xi\rangle^{2\rho}}{\langle\xi\rangle^{2\rho}\langle\tau-\xi^2\rangle^{2b-\rho}}\int\frac{\chi_{\{\xi_1;\exists \tau_1; (\xi_1,\tau_1)\in C(\xi,\tau)\}d\xi_1}}{\langle\tau-\xi^2+2\xi\xi_1\rangle^{4b-1}}.\label{rio2}
\end{eqnarray}
Set $\eta=\tau-\xi^2+2\xi\xi_1$, then $d\eta=2\xi d \xi_1$ and $$|\eta|=|\tau_1+\xi_1^2+\tau-\tau_1-(\xi-\xi_1)^2|\leq 2 |\tau-\xi^2|,\  \text{in}\  B(\xi,\tau)_{\xi_1}.$$ Substituting into \eqref{rio2}, we see that \eqref{rio2} is bounded by 
\begin{equation*}
	\sup_{|\xi|>1, \tau\in \mathbb{R}}\frac{c\chi_{\{|\xi|>1\}}}{\langle\xi\rangle\langle\tau-\xi^2\rangle^{2b-\rho}}\int_{|\eta|\leq 2|\tau-\xi^2|}\frac{d\eta}{\langle\eta\rangle^{4b-1}} \leq   \sup_{|\xi|>1, \tau\in \mathbb{R}}\frac{c}{\langle \tau-\xi^2\rangle^{6b-\rho-2}},
\end{equation*}
which is bounded provided  and $\rho\leq 6b-2$. 

Finally, we shall obtain \eqref{18094}. By Lemma \ref{lemagtv} we see that left hand side of\eqref{18094} is bounded by
\begin{equation*}
 \sup_{|\xi_1|>10,\tau_1\in\mathbb{R}}\frac{c\chi_{\{|\xi_1|\geq 10\}}}{\langle\tau_1-\xi_1^2\rangle^{2b}}\int_{D(\xi_1,\tau_1)_{\xi}} \frac{(|\xi_1||\xi-\xi_1|)^{2\rho}d\xi}{\langle \xi\rangle^{2\rho}\langle \tau_1-\xi_1^2-2\xi\xi_2\rangle^{4b-1}},
\end{equation*}
where $D(\xi_1,\tau_1)_{\xi}=\{\xi;\ \exists \  \tau; (\xi,\tau)\in D(\xi_1,\tau_1)\}$.

Now we subdivide $D(\xi_1,\tau_1)_{\xi}$ into two pieces, i.e., $D_1(\xi_1,\tau_1)=\{\xi\in D(\xi_1,\tau_1)_{\xi};\, |\xi_1|\leq 100|\xi| \}$ and $D_2(\xi_1,\tau_1)=\{\xi\in (D(\xi_1,\tau_1))_{\xi};\, |\xi_1|> 100|\xi|,\ |\xi_1|\leq 500|\tau_1-\xi_1^2| \}$.

In $D_1(\xi_1,\tau_1)$ we use the algebraic relation
\begin{equation*}
(\tau-\xi^2)-(\tau_1-\xi_1^2)-(\tau-\tau_1+(\xi-\xi_1^2))=-2\xi(\xi-\xi_1)
\end{equation*}
and Lemma \ref{lemagtv} to obtain

\begin{equation*}
\sup_{|\xi_1|>10,\tau_1\in\mathbb{R}}\frac{\chi_{\{|\xi_1|>10\}}}{\langle\tau_1-\xi_1^2\rangle^{2b}}\int \frac{\chi_{D_1}(|\xi_1||\xi-\xi_1|)^{2\rho}d\xi}{\langle \xi\rangle^{2\rho}\langle \tau_1-\xi_1^2-2\xi\xi_2\rangle^{4b-1}} \leq \sup_{|\xi_1|>10,\tau_1\in\mathbb{R}}c \int \frac{\chi_{D_1}d\xi}{\langle \tau_1-\xi_1^2-2\xi\xi_2\rangle^{4b-1}},
\end{equation*}
which is bounded provided $b>\frac{3}{8}$ and $\rho\leq b$.

In $D_2$ we have $|\xi||\xi-\xi_1|\leq c|\xi_1|^2\leq c|\tau_1-\xi_1^2|^2$, it follows that the integral in this region can be controlled by
\begin{equation}\label{2209}
\sup_{|\xi_1|>10,\tau_1\in\mathbb{R}}\frac{c|\xi_1|^{4\rho}\chi_{|\xi_1|>10}}{\langle\tau_1-\xi_1^2\rangle^{2b}}\int \frac{\chi_{D_2}d\xi}{\langle \xi\rangle^{2\rho}\langle \tau_1-\xi_1^2-2\xi\xi_2\rangle^{4b-1}},
\end{equation}
By Lemma \ref{lemanovo}, the expression \eqref{2209} is bounded if $b>\frac{3}{8}$ and $\rho<\frac{b}{2}$, and hence the proof of Lemma \ref{estimativan2} is completed.
\subsection{Proof of Lemma \ref{estimativan3}}
To get the part (a) it suffices to show that
\begin{equation}\label{23092}
\left\|\frac{1}{\langle\tau-\xi^2\rangle^{2b}}\int\int\frac{d\xi_1d\tau_1}{\langle\tau_1+\xi_1^2\rangle^{2b}\langle\tau-\tau_1+(\xi-\xi_1)^2\rangle^{2b}}\right\|_{L_{\xi,\tau}^{\infty}}\leq c,
\end{equation}
\begin{equation}\label{23093}
\left\|\frac{1}{\langle\tau-\xi^2\rangle^{2b}\langle \xi\rangle^{2\rho}}\int\int_{A_3}\frac{ (|\xi_1||\xi-\xi_1|)^{2\rho}d\xi_1 d\tau_1}{\langle\tau_1+\xi_1^2\rangle^{2b}\langle\tau-\tau_1+(\xi-\xi_1)^2\rangle^{2b}}\right\|_{L_{\xi,\tau}^{\infty}}\leq  c,
\end{equation}
\begin{equation}\label{23094}
\left\|\frac{\chi_{\{|\xi_1|>10\}}}{\langle\tau_1+\xi_1^2\rangle^{2b}} \int\int_{B_3}\frac{(|\xi_1||\xi-\xi_1|)^{2\rho}d\tau d\xi}{\langle\xi\rangle^{2\rho}\langle\tau-\xi^2\rangle^{2b}\langle\tau-\tau_1+(\xi-\xi_1)^2\rangle^{2b}}\right\|_{L_{\xi_1,\tau_1}^{\infty}}\leq c,
\end{equation}
for $\rho\in (\frac{1}{2},\frac{3}{4})$, where 
\begin{equation*}
A_3=A_3(\xi,\tau)=\{(\xi_1,\tau_1)\in \R^2;\,\max \{|\tau-\tau_1+(\xi-\xi_1)^2|,|\tau_1+\xi_1^2|,|\tau-\xi^2|\}= |\tau-\xi^2|  \},
\end{equation*}
\begin{equation*}
B_3=B_3(\xi_1,\tau_1)=\{ (\xi,\tau)\in\mathbb{R}^2;\, \max \{|\tau-\tau_1+(\xi-\xi_1)^2|,|\tau_1+\xi_1^2|, |\tau-\xi^2|\}=|\tau_1+\xi_1^2| \}.
\end{equation*}

The estimate \eqref{23092} can be obtained by using the same argument given in the estimate \eqref{01121} in the proof of Lemma \ref{objetivo} (part (a)). The estimate \eqref{23093} and \eqref{23094} can be obtained by following the ideas used in the estimates \eqref{eq2} and \eqref{eq3} in the proof of Lemma \ref{objetivo} (part (a)) and by using the following algebraic relation 

\begin{equation}\label{2901}
\tau-\tau_1+(\xi-\xi_1)^2+\tau_1+\xi_1^2-(\tau-\xi^2)=(\xi-\xi_1)^2+\xi_1^2+\xi^2.
\end{equation}
To obtain the part (b) it suffices to prove that
\begin{equation}\label{1442}
\left\|\frac{\chi_{\{10|\tau|\leq |\xi|^2\}}}{\langle\tau-\xi^2\rangle^{2b}\langle \tau\rangle^{\rho}}\int\int_{A_3}\frac{ (|\xi_1||\xi-\xi_1|)^{2\rho}d\tau_1 d\xi_1}{\langle\tau_1+\xi_1^2\rangle^{2b}\langle\tau-\tau_1+(\xi-\xi_1)^2\rangle^{2b}}\right\|_{L_{\xi,\tau}^{\infty}}\leq c,
\end{equation}
\begin{equation}\label{1443}
	\left\|\frac{1}{\langle\tau+\xi_1^2\rangle^{2b}}\int\int_{B_3}\frac{\chi_{\{10|\tau|\leq |\xi|^2\}}(|\xi_1||\xi-\xi_1|)^{2\rho}d\tau d\xi}{\langle\tau\rangle^{\rho}\langle \tau-\xi^2\rangle^{2b}\langle \tau-\tau_1+(\xi-\xi_1)^2\rangle^{2b}}\right\|_{L_{\xi_1,\tau_1}^{\infty}}\leq c,
\end{equation}
where $A_3=A_3(\xi,\tau)=\{ (\xi_1,\tau_1)\in\mathbb{R}^2;\, \max \{|\tau-\tau_1+(\xi-\xi_1)^2|,|\tau_1+\xi_1^2|, |\tau-\xi^2|\}=|\tau-\xi^2| \}$ and
$B_3=B_3(\xi_1,\tau_1)=\{ (\xi,\tau)\in\mathbb{R}^2;\, \max \{|\tau-\tau_1+(\xi-\xi_1)^2|,|\tau_1+\xi_1^2|, |\tau-\xi^2|\}=|\tau_1+\xi_1^2| \}$.

These estimates can be obtained by using a similar argument to that described in the proof of Lemma \ref{objetivo} (part (b)). The detailed computation is omitted here.
\section{Proof of Theorem \ref{teorema}}\label{section8}
We will prove Theorem \ref{teorema} for the nonlinearity $N_1(u,\overline{u})=u^2$, the proofs for the nonlinearities $N_2(u,\overline{u})=|u|^2$ and $N_3(u,\overline{u})=\overline{u}^2$ follow the same ideas, by using Lemmas \ref{estimativan2} and \ref{estimativan3}. 
 
Using a scaling argument, we can assume 
$$\|u_0\|_{H^s(\mathbb{R}^+)}+\|f\|_{H^{\frac{2s+1}{4}}(\mathbb{R}^+)}=\delta,$$
for $\delta$ sufficiently small. Indeed, $u$ solves the problem (\ref{SK}) if and only if for $ \lambda>0$, $u_{\lambda}(x,t)=\lambda^2u(\lambda x,\lambda^2 t)$ solves the problem (\ref{SK}) with initial-boundary conditions $u_{\lambda}(x,0)=\lambda^2u_0(\lambda x)$ and $u_{\lambda}(0,t)=\lambda^2f(\lambda^2 t)$.
Thus for $s\leq 0$, $\|u_{\lambda}(\cdot,0)\|_{H^s(\mathbb{R}^+)}\leq \lambda^{3/2+s} \|u(\cdot,0)\|_{H^{s}(\mathbb{R}^+)}$ and $\|u_{\lambda}(0,\cdot)\|_{H^{\frac{2s+1}{4}}(\mathbb{R}^+)}\leq \lambda^{3/2+s} \|u(0,\cdot)\|_{H^{\frac{2s+1}{4}}(\mathbb{R}^+)}.$

Select an extension $\tilde{u}_0\in H^s(\mathbb{R})$ of $u_0$ such that $\|\tilde{u}_0\|_{H^s(\mathbb{R})}\leq c\|u_0\|_{H^s(\mathbb{R}^+)}$. Let $b=b(s)<\frac{1}{2}$ such that the estimates (\ref{objetivo}) and (\ref{objetivoabc}) are valid .

Set $Z$ the Banach Space given by $
	Z=C\big(\mathbb{R}_t;\,H^s(\mathbb{R}_x)\big)\cap C\big(\mathbb{R}_x;\,H^{\frac{2s+1}{4}}(\mathbb{R}_t)\big)\cap X^{s,b}.$

Let
\begin{equation*}
	\Lambda(u)(t)=\psi(t)e^{-it\partial_x^2}\tilde{u}_0+\psi(t)\mathcal{D}(u^2)(t)+\psi(t)\mathcal{L}^{\lambda}h(t),
\end{equation*}
where $h(t)=e^{-i\frac{3\lambda \pi}{4}}\big(\chi_{(0,+\infty)}\psi(t)f(t)-\psi(t)e^{-it\partial_x^2}\tilde{u}_0\big|_{x=0}-\psi(t)\mathcal{D}(u^2)(t)\big|_{x=0}\big)\big|_{(0,+\infty)}$ and $\lambda=\lambda(s)$ is such that $-1<\lambda<\frac{1}{2}$ and $s-\frac{1}{2}<\lambda<s+\frac{1}{2}$.
 
Initially, we show that $\mathcal{L}^{\lambda}h(t)$ is well defined. By Lemmas \ref{sobolevh0}, \ref{sobolev0} , \ref{grupo}, \ref{duhamel}, \ref{objetivo} and \ref{objetivoabc} we have
\begin{eqnarray}
	\|h\|_{H^{\frac{2s+1}{4}}(\mathbb{R}^+)}
	&\leq&\big\|\chi_{(0,+\infty)}(\psi(t)f(t)-\psi(t)e^{-it\partial_x^2}\tilde{u}_0|_{x=0}-\psi(t)\mathcal{D}(u^2)(t)\big|_{x=0})\big\|_{H^{\frac{2s+1}{4}}(\mathbb{R})}\nonumber\\
	&\leq & c \big(\|f\|_{H^{\frac{2s+1}{4}}(\mathbb{R}^+)}+\|\tilde{u}_0\|_{H^s(\mathbb{R})}+\big\|u^2(x,t)\big\|_{X^{s,-b}}+c_1(s)\big\|u^2(x,t)\big\|_{W^{s,-b}}\big)\nonumber\\
	&\leq & c \big(\|f\|_{H^{\frac{2s+1}{4}}(\mathbb{R}^+)}+\|\tilde{u}_0\|_{H^s(\mathbb{R})}+\|u\|_{X^{s,b}}^2\big),\label{demelo}
\end{eqnarray}
where $c_1(s)=0$, if $s\in(-1/2,0]$ e $c_1(s)=1$, if $s\in (-3/4,-1/2]$. 

Since $u\in Z$ we obtain $h\in H^{\frac{2s+1}{4}}(\mathbb{R}^+)$. If $-\frac{3}{4}<s\leq 0$, then $-\frac{1}{8}\leq\frac{2s+1}{4}\leq \frac{1}{4}$, and Lemma \ref{sobolevh0} shows that $h\in H_0^{\frac{2s+1}{4}}(\mathbb{R}^+)$. Thus, $\mathcal{L}^{\lambda}h(t)$ is well defined and $	\|h\|_{H^{\frac{2s+1}{4}}(\mathbb{R}^+)}\sim	\|h\|_{H_0^{\frac{2s+1}{4}}(\mathbb{R}^+)}$.

Our goal is to show that $\Lambda$ defines a contraction map on any ball of $Z$.

Using Lemmas \ref{grupo}, \ref{duhamel} and \ref{estimativan1} we see that
\begin{equation*}
	\big\|\psi(t)e^{-it\partial_x^2}\tilde{u}_0\big\|_{Z}\leq c\|\tilde{u}_0\|_{H^{s}(\mathbb{R})}\leq c \|u_0\|_{H^{s}(\mathbb{R}^+)}\ \text{and}\; \;\big\|\psi(t)\mathcal{D}(u^2)(t)\big\|_{Z}\leq c\|u\|_{X^{s,b}}^2.
\end{equation*}
Combining Lemma \ref{edbf} and expression \eqref{demelo} we obtain
\begin{equation*}
	\big\|\psi(t)\mathcal{L}^{\lambda}h(t)\big\|_{Z}\leq c\|h\|_{H_0^{\frac{2s+1}{4}}(\mathbb{R}^+)}\\
	\leq c\big(\|f\|_{H^{\frac{2s+1}{4}}(\mathbb{R}^+)}+\|u_0\|_{H^s(\mathbb{R}^+)}+\|u\|_{Z}^2\big),
\end{equation*}
for $-1<\lambda<\frac{1}{2}$ and $s-\frac{1}{2}<\lambda<s+\frac{1}{2}$. Note that for $s>-\frac{1}{2}$, we can take $\lambda=0$ and as $s>-\frac{3}{4}$ we can choice adequate $\lambda=\lambda(s)$.
Thus, we obtain
\begin{equation}\label{humberto1}
	\|\Lambda u\|_{Z}\leq c(\|f\|_{H^{\frac{2s+1}{4}}(\mathbb{R}^+)}+\|u_0\|_{H^s(\mathbb{R}^+)}+\|u\|_{Z}^2).
\end{equation}
Similarly, we obtain for $u$ and $v$ in $Z$,
 
\begin{equation}\label{humberto2}
	\|\Lambda u-\Lambda v\|_{Z}\leq c^2(\|u\|_{Z}+\|v\|_{Z})\|u-v\|_{Z}.
\end{equation}
Let $B_{s,b}(c\delta)=\{u\in Z;\|u\|_{Z}\leq 2c\delta \}$.  If $u\in B_{s,b}(c\delta)$, then (\ref{humberto1}) implies
\begin{equation}
	\|\Lambda u\|_{Z}\leq c\delta+c(2c\delta)^2\leq 2c\delta.
\end{equation}
Now we choice $\delta$, such that $4c^2\delta\leq \frac{1}{2}$. By (\ref{humberto2}) if $u$ and $v$ in $B_{s,b}(c\delta)$ we have
\begin{equation*}
	\|\Lambda u-\Lambda v\|_{Z}\leq 4c^2\delta\|u-v\|_{Z}\leq \frac{1}{2}\|u-v\|_{Z}.
\end{equation*}
Thus, $\Lambda$ defines a contraction on $B_{s,b}(c\delta)$ and consequently there exists a unique function $\tilde{u}\in B_{s,b}(c\delta)$ such that $\tilde{u}=\Lambda(\tilde{u})$. Therefore $u(x,t):=\tilde{u}(x,t)|_{{\mathbb{R}^+\times[0,1]}}$ solves the IBVP \eqref{SK} with nonlinearity $N_1(u,\overline{u})=u^2$, in the time interval $[0,1]$.
\subsection*{Acknowledgments}This paper is part of my Ph.D. thesis at the Federal University of Rio de Janeiro under the guidance of my advisor Ad\'an Corcho. I want to take the opportunity to express my sincere gratitude to him. I would also like to thank the referee for a careful reading and helpful suggestions. The author was partially supported by CAPES.


\begin{thebibliography}{99}


\bibitem{bt} {I. Bejenaru and T. Tao, } {\it Sharp well-posedness and ill-posedness results for a quadratic non-linear Schr{\"o}dinger equation},  J. Funct. Anal., Vol.233 (2006), 228-259.
\bibitem{bop}D. Bekiranov, T. Ogawa, and G. Ponce, {\it Weak solvability and well-posedness of
a coupled Schr\"odinger-Korteweg-de Vries equation for capillary-gravity wave interactions}, Proc. Amer. Math. Soc., Vol.125 (1997), 2907-2919.


\bibitem{Bourgain1}{J. Bourgain,}
{\it Fourier transform restriction phenomena for certain lattice subsets and applications
to nonlinear evolution equations. I. Schr\"odinger equations,}
Geom. Funct. Anal., Vol.3 (1993), 107-156.




\bibitem{Bona}{J. L. Bona, S. M. Sun, and B. Y. Zhang,}
{\it Boundary smoothing properties of the Korteweg-de Vries equation in a quarter plane and applications,}
{Dyn. Partial Differ. Equ., Vol.3 (2006), 1-70.}

\bibitem{Bona1} {J. L. Bona, S. M. Sun, and B. Y. Zhang}, {Nonhomogeneous Boundary-Value Problems for One-Dimensional Nonlinear Schr\"odinger Equations,} preprint arXiv:1503.00065.






 
   
   \bibitem{CW} T. Cazenave and F.Weissler, {\it The Cauchy problem for the critical nonlinear Schr\"odinger equation in $H^s$}, Nonlinear Anal., Vol.14 (1990), 807-836.

\bibitem{CK} J. E. Colliander and C. E. Kenig, {\it The generalized Korteweg-de Vries equation on the half
	line}, Comm. Partial Differential Equations, Vol.27 (2002), 2187-2266.

\bibitem{CM} A. J. Corcho and M. Cavalcante, {\it The initial boundary value problem for the Schr\"ondiger-Korteweg-de Vries system on the half-line}, preprint, arxiv: 1608.06308.


\bibitem{tzirakis} {M. B. Erdo\u{g}an, N. Tzirakis,} {\it Regularity properties of the cubic nonlinear Schr\"odinger equation on the half-line},
J. Funct. Anal., Vol.233 (2016), 2539-2568.


















\bibitem{Fr}{F. G. Friedlander,}{ \textquotedblleft Introduction to the theory of distributions\textquotedblright}, second ed., Press, Cambridge, 1998, With additional material by M. Joshi.



\bibitem{GTV2}  J. Ginibre and G. Velo, {\it Scattering theory in the energy space for a class of nonlinear Schr\"odinger equations}, J. Math. Pures Appl., Vol.64 (1985), 363-401.

\bibitem{GTV} {J. Ginibre, Y. Tsutsumi, and G. Velo,} {\it On the Cauchy problem for the Zakharov system}, J. Funct. Anal., Vol.151 (1997), 384-436.




\bibitem{Holmer} {J. Holmer,} {\it The initial-boundary value problem for the 1D nonlinear Schr\"odinger equation on the half-line}, Differential and Integral Equations, Vol.18 (2005), 647-668.

\bibitem{Holmerkdv} {J. Holmer,} {\it The initial-boundary value problem for the Korteweg-de Vries equation}, Communications in Partial Differential Equations, Vol.31 (2006), 1151-1190.


 \bibitem{Kato} T. Kato, {\it On nonlinear Scr\"odinger equations. II.  $H^s$- solutions and unconditional
 	well-posedness,}, J. d’Analyse Math., 67 (1995), 281–306.
 
\bibitem{KPV} C. E. Kenig, Gustavo Ponce, and Luis Vega, {\it Oscillatory integrals and regularity of dispersive equations}, Indiana Univ. Math. J., Vol.40 (1991), 33-69.

\bibitem{KPV2-a}{C. E. Kenig, G. Ponce, and L. Vega,}
{\it Quadratic Forms for the 1-D semilinear Schr\"odinger equation,}
Trans. Amer. Math. Soc., Vol.348 (1996),  3323-3353.






\bibitem{Kishimoto} {N. Kishimoto,} {\it Local well-posedness for the Cauchy problem of the quadratic Schr\"odinger equation with nonlinearity $\overline{u}^2$}, Commun. Pure Appl. Anal., Vol.7 (2008), 1123-1143.



\bibitem{LP}{F. Linares and G. Ponce,}{\textquotedblleft Introduction to Nonlinear Dispersive Equations\textquotedblright}, second ed., Springer, 2014.




\bibitem{T} {Y. Tsutsumi,} {\it $L^2$-solutions for nonlinear Schr\"odinger equations and nonlinear groups},
Funkcial. Ekvac., Vol.30 (1987), 115-125. 



\end{thebibliography}
\end{document}